\documentclass[12pt]{amsart}
\usepackage{amsmath,amssymb,amsfonts} 
\usepackage{graphics}                 
\usepackage[square, numbers]{natbib}
\usepackage[english]{babel}
\usepackage{tikz} 
\usepackage[utf8]{inputenc}
\usepackage[all]{xy}	
\usepackage[T1]{fontenc}
\usepackage{pifont}
\usepackage[margin=1.2in]{geometry}

\newcommand{\C}{\mathbb{C}}

\newcommand{\DCF}{\mathrm{DCF}}
\newcommand{\BE}{\mathrm{BE}}

\usepackage{hyperref}

\theoremstyle{definition}
\newtheorem{defn}{Definition}[section]

\newtheorem{exam}[defn]{Example}

\theoremstyle{plain}
\newtheorem{Prop}[defn]{Proposition}
\newtheorem{theorem}[defn]{Theorem}

\newtheorem{cor}[defn]{Corollary}
\newtheorem{lem}[defn]{Lemma}

\newtheorem{fact}[defn]{Fact}
\newtheorem{Assumption}[defn]{Assumption}

\newtheorem*{claim}{Claim}

\newtheorem{notation}[defn]{Notation}

\theoremstyle{exampstyle}\newtheorem{thmx}{Theorem}

\theoremstyle{exampstyle}

\newtheorem*{corx}{Corollary}

\def\Ind{\setbox0=\hbox{$x$}\kern\wd0\hbox to 0pt{\hss$\mid$\hss}
\lower.9\ht0\hbox to 0pt{\hss$\smile$\hss}\kern\wd0}
\def\Notind{\setbox0=\hbox{$x$}\kern\wd0\hbox to 0pt{\mathchardef
\nn=12854\hss$\nn$\kern1.4\wd0\hss}\hbox to
0pt{\hss$\mid$\hss}\lower.9\ht0 \hbox to
0pt{\hss$\smile$\hss}\kern\wd0}
\def\ind{\mathop{\mathpalette\Ind{}}}

\newcommand{\Gm}{\mathrm{G}_m}

\newcommand{\Ga}{\mathrm{G}_a}

\title{Integration in finite terms and exponentially algebraic functions}
\author{Rémi Jaoui}
\address{Institut Camille Jordan, Université Lyon 1 - CNRS, Villeurbanne, France}
\email{jaoui@math.univ-lyon1.fr}

\author{Jonathan Kirby}
\address{School of Engineering, Mathematics and Physics, University of East Anglia, Norwich, UK}
\email{jonathan.kirby@uea.ac.uk}
 
\newtheoremstyle{exampstyle}
{3pt} 
{3pt} 
{\itshape} 
{} 
{\bfseries} 
{.} 
{.5em} 
{} 

\makeindex
\begin{document}
\maketitle
\begin{abstract}
We develop techniques at the interface between differential algebra and model theory to study the following problems of exponential algebraicity:
\begin{itemize}
\item[(1)] Does a given algebraic differential equation admits an {\em exponentially algebraic solution}, that is, a holomorphic solution which is definable in the structure of restricted elementary functions?
\item[(2)] Do solutions of a given list of algebraic differential equations share a nontrivial {\em exponentially algebraic relation}, that is, a nontrivial relation definable in the structure of restricted elementary functions?
\end{itemize}
These problems can be traced back to the work of Abel and Liouville on the problem of integration in finite terms. This article concerns generalizations of their techniques adapted to the study of exponential transcendence and independence problems for  more general systems of differential equations. 

As concrete applications, we obtain exponential transcendence and independence statements for several classical functions: the error function, the Bessel functions, indefinite integrals of algebraic expressions involving Lambert's $W$-function, the equation of the pendulum, as well as  corresponding decidability results.
\end{abstract}
 \setcounter{tocdepth}{1}

\tableofcontents

\section{Introduction}

In this article, we study the following ``functional'' problems of exponential algebraicity. Consider the o-minimal structure of {\em restricted elementary functions} $$\mathbb{R}_{\mathrm{RE}} := (\mathbb{R}, 0,1,+,\times, < , \text{restricted complex exp})$$  obtained by expanding semialgebraic geometry by the complex exponential function restricted to some (equivalently all) compact box \cite{vandenDries}. 
\begin{itemize}
\item[(1)] Does a given algebraic differential equation admits an  {\em exponentially algebraic solution}, that is, a holomorphic solution which is definable in the structure $\mathbb{R}_{\mathrm{RE}}$ of restricted elementary functions?
\item[(2)] Do solutions of a given list of algebraic differential equations share a nontrivial {\em exponentially algebraic relation}, that is, a nontrivial relation definable in the structure $\mathbb{R}_{\mathrm{RE}}$ of restricted elementary functions?
\end{itemize}
They are the natural respective analogues of the problem of finding the algebraic solutions and the algebraic relations between solutions of algebraic differential equations for the structure $\mathbb{R}_{\mathrm{sa}} = (\mathbb{R}, 0,1,+,\times, <)$. It is worth stressing that despite the appearance of the field of real numbers, the problems (1) and (2) are concerned with the {\em holomorphic solutions} of complex algebraic equations and hence with the definability properties of holomorphic functions in certain reducts of the o-minimal structure $\mathbb{R}_{an}$ of restricted analytic functions \cite{Wilkie}.  

In the {\em semialgebraic} case, decidability of (1) is a famously open problem known as Poincaré's problem \cite{Poincare}. For linear differential equations, the problems (1) and (2) are the subject of several important number-theoretic conjectures and {\em Picard-Vessiot differential Galois theory} \cite{vdP-Singer} provides effective tools to study linear versions of these problems in the semialgebraic case. The aim of this article is to develop a parallel Galois-theoretic framework, Theorem \ref{intro-thmA} and Theorem \ref{intro-thmB} below, to study the analogous problems of exponential algebraicity. 

Applications of this framework to the study of classical functions which already appears in Liouville's work \cite{Liouville2,Liouville} are discussed in the introduction and described in greater details in the fifth section of the text. 

\subsection{Some history} The study of the problems (1) and (2) in $\mathbb{R}_{\mathrm{RE}}$ can be traced back to the problem of integration in finite terms studied by Abel and Liouville. The classical framework focuses on the case of indefinite integrals and on a {\em subclass} of the class of exponentially algebraic functions, the class of {\em elementary functions}.  They are the analytic branches of the meromorphic functions which can obtained by iterated compositions of the complex exponential, the branches of the complex logarithm and algebraic functions involving possibly several complex variables.

A predecessor of the larger class of {\em exponentially algebraic functions} is what Liouville calls ``les fonctions finies implicites'' in \cite{Liouville}. They are the functions which can be implicitly defined from the class of elementary functions by ``using the implicit function theorem''. We refer to the main theorem of \cite{Wilkie} for a precise version of this equivalence. A prototype studied by Liouville is the Lambert $W$-function given by
\begin{equation}\label{equation-Lambert}
W(z) = \sum_{m = 1}^\infty \frac{(-m)^{m-1}}{m!} z^m
\end{equation}
which appears in the solutions of many problems in enumerative combinatorics, biology and physics \cite{Lambert-general,Lambert-quantum}. The ubiquity of (\ref{equation-Lambert}) comes from  its definition as a solution of the functional equation $w \cdot e^{w} =  z$ which provides its implicit definition as an exponentially algebraic function. Liouville \cite{Liouville} showed however that (\ref{equation-Lambert}) can not be explicitly defined within the smaller class of elementary functions. \\

The starting point of Liouville's theory of integration \cite{Liouville2} motivated by the integration of the pendulum concerns the {\em exponential transcendence} of the elliptic integrals:

\begin{equation}\label{equation-elliptic-integral-intro}
 \mathrm{sn}(z) = \int \frac{dz}{\sqrt{z^3 + az + b}} \text{ where } 4a^3 + 27b^2 \neq 0.
 \end{equation}
Liouville showed that this function is not elementary. As described in Ritt's monograph \cite{Ritt-book} on Liouville's theory and in \cite{Singer-survey}, the fact that (\ref{equation-elliptic-integral-intro}) is not even exponentially algebraic then follows from a statement written by Abel a few years before, which in the terminology of this paper can be reformulated as:
\begin{quote} 
\textit{if the indefinite integral of an algebraic function is exponentially algebraic then it is an elementary function}.
\end{quote}
Although there is no trace of Abel's proof of this statement, Ritt \cite{Ritt-book} and then Risch \cite{Risch3} have both established rigorous forms of Abel's statement. 

Since the seventies, the scope of Liouville's theory of integration has been greatly expanded in connection with Ax's proof of the functional version of Schanuel conjecture \cite{Ax-Theorem}. The connection between these two problems already appears implicitly in Risch's ``solution of the problem of integration in finite terms'' \cite{Risch2} and was made fully explicit by Rosenlicht in \cite{Rosenlicht}. These developments led to the development of complete decision procedures \cite{Risch2, Bronstein, Singer-Raab,Pila-Tsimerman} to decide if {\em the indefinite integral of an arbitrary elementary function} $f(z)$ is elementary and to compute it explicitly when it is.

\subsection{A generalization of Abel's result} The first result of this article is a generalization of Abel's result which is missing from Liouville's theory of integration. For instance, it is known since Liouville's work that the {\em error function}
$$\mathrm{erf}(z) = \int e^{-z^2} dz$$ 
is not elementary but is it exponentially algebraic? In other words, can it be {\em implicitly} constructed from the class of elementary functions? This strengthening is in general not automatic. For instance, all the solutions of the first-order equation associated to the Lotka-Volterra systems:
\begin{equation}\label{Lotka-intro}
\frac {dy} {dx} = \frac{\alpha x - \beta xy}{\gamma y - \delta xy} \text{ where } \alpha, \beta,\delta,\gamma \in \mathbb{C}^\ast \text{ and } \alpha \neq \gamma
\end{equation}
are exponentially algebraic but (\ref{Lotka-intro}) admits no elementary solutions, see  \cite{Duan-Nagloo,DEJ}. The following theorem provides a general dividing line which separates the algebraic differential equations which satisfy Abel's statement from the other ones.

\begin{thmx}[Theorem \ref{theoremA}] \label{intro-thmA}
Consider an algebraic differential equation 
\begin{equation} \label{equation-TheoremA} 
F(y,y',\ldots, y^{(n)}) = 0
\end{equation}
whose coefficients are elementary functions and assume that the equation is {\em internal to the constants}. Then every exponentially algebraic solution of the equation is an elementary function (in the sense of Liouville).
\end{thmx} 

The assumption that the equation (\ref{equation-TheoremA}) is internal to the constants is an important model-theoretic assumption in differential algebra which allows for the development of a differential Galois theory \cite{Hrushovski-Galois,Pillay-Galois, Pillay-Leon-Sanchez}. The differential equation (\ref{Lotka-intro}) witnesses that it can not be avoided. An algebraic differential equation is {\em internal to the constants} if from finitely many particular solutions, one can express any other solution (in any differential field extension) as an algebraic combination of these particular solutions, their derivatives and constants. In particular, this theorem applies to many classical families of examples such as  {\em for the indefinite integrals of elementary functions, for the solutions of (possibly inhomogeneous) linear differential equations, Riccati equations and elliptic differential equations}.  

Using a combination of classical results and techniques from the problem of integration in finite terms, Theorem A implies the exponential transcendence of several important classical functions such as the expressions of the form
\begin{equation}\label{equation-intro-linear examples}
 \mathrm{Ei}(z) = \int \frac {e^z} {z} dz,  \mathrm{Ai}(z) = \frac 1 \pi \int_0^\infty \mathrm{cos}(t^3/3 + zt) dt \text{ and } \int\frac{W(z)}{z^2} dz.
 \end{equation}
The first two examples can be studied inside decidable frameworks: Risch's solution of the problem of integration in finite terms \cite{Singer-Raab} and Hrushovski's computation of differential Galois groups \cite{Hrushovski-Galois}. Corollary \ref{cor-Lambert} explains how to extend Risch's decidable framework to cover certain indefinite integrals involving local analytic inverses of elementary functions such as the last example in (\ref{equation-intro-linear examples}).

Another meaningful application of Theorem \ref{intro-thmA} is the exponential transcendence of the indefinite elliptic integrals also proved in \cite{Jones-Kirby-Servi} studied in our setting using  elliptic differential equations. 

\begin{corx}[Corollary \ref{cor-pendulum}]
Consider $\omega \in \mathbb{R}^\ast$ and the equation of the pendulum 
\begin{equation}\label{equation-intro-pendulum}
\theta'' + \omega^2 \cdot \mathrm{sin}(\theta) = 0.
\end{equation}
The exponentially algebraic solutions of (\ref{equation-intro-pendulum}) are exactly the solutions with energy $$H = 1/2 (\theta')^2 - \omega^2 \mathrm{cos}(\theta) = \pm \omega^2.$$
\end{corx}

The restriction to the real locus of the exponentially algebraic solutions of (\ref{equation-intro-pendulum}) are easy to detect on the usual phase portrait of  the pendulum: they are the so-called {\em homoclinic solutions} which separate the bounded solutions oscillating around the stable equilibrium positions from the unbounded ones. They can indeed be expressed using inverse trigonometric functions \cite{Chenciner}. 

We give two proofs of this result. The first proof based on Theorem \ref{intro-thmA} provides a {\em decision procedure} for finding the exponential algebraic solutions of  one-dimensional Hamiltonian systems of the form
\begin{equation}
H = \frac {y^2} 2 + V(x)\text{ where } V(x) \text{ is an elementary function.}
\end{equation}
This appears as Corollary \ref{cor-Hamiltonian} below. This procedure involves one application of the implicit function theorem and the computation of an indefinite integral. In the case of (\ref{equation-intro-pendulum}), the computation reduces to the computation of the Legendre's elliptic integral which were already known to Liouville.

The second Galois-theoretic proof uses Theorem \ref{intro-thmB} below. Using the change of variables $u = \mathrm{cos}(\theta/2)$ which preserves exponential algebraicity, (\ref{equation-intro-pendulum}) can be transformed into an autonomous algebraic differential equation given  as a {\em nonisotrivial family of elliptic differential equations} 
\begin{equation}\label{equation-intro-elliptic}
\Big(\frac{du}{dz}\Big)^2 = \omega^2(1-u^2)\Big( u^2 - \frac{\omega^2-h_0}{2\omega^2}\Big).
\end{equation}
The ``nondegeneracy condition'' $h_0 \neq \pm \omega^2$ describe the smooth fibers of this elliptic fibration, that is, those which do not degenerate as rational curves. Under this condition, (\ref{equation-intro-elliptic}) is an {\em elliptic differential equation}, that is, a differential equation which is internal to the constants and with Galois group an elliptic curve.  In fact, Theorem \ref{intro-thmB} leads to a stronger result namely that: 

\begin{quote} 
whenever $(\omega_i,h_i)$ are chosen so that $h_i \neq \pm \omega_i^2$ and the elliptic curves
$$ v^2 =  \omega_i^2(1-u^2)\Big(u^2 - \frac{\omega_i^2-h_i}{2\omega_i^2}\Big)$$
are pairwise nonisogeneous then corresponding solutions of (\ref{equation-intro-pendulum}) are exponentially algebraically independent over $\mathbb{C}(z)$.
\end{quote} 
The content of such exponential algebraic independence statements is described below before the statement of Theorem B.

\subsection{Exponential algebraic independence and differential Galois theory} \text{ } We say that meromorphic functions $f_1,\ldots, f_n \in \mathcal M(U)$ are {\em exponentially algebraically independent over $\mathbb{C}(z)$} if they satisfy no nontrivial exponentially algebraic relations in the sense that: 
\begin{quote} 
whenever $G$ is a holomorphic function definable on a connected open subset of $\mathbb{C}^{n+1}$ in $\mathbb{R}_{\mathrm{RE}}$ such that the function  
$z \mapsto G(z,f_1(z),\ldots, f_n(z))$  is well-defined and vanishes on a nontrivial open subset of $U$ then $G$ vanishes identically on $\mathbb{C}^{n+1}$. 
\end{quote} 
Natural problems of exponential algebraicity can be built from the same classical examples studied by Liouville: For instance, are $\mathrm{sn}(z)$ and $\mathrm{erf}(z)$ exponentially algebraically independent over $\mathbb{C}(z)$? What about $\mathrm{Ai}(z)$ and $\mathrm{Ai}'(z)$?  \\

We describe of framework using definable Galois theory \cite{Pillay-Galois,Pillay-Leon-Sanchez} which applies to (complete) types which are internal to the constants in the theory $\mathrm{DCF}$ of differentially closed fields. This generalizes Picard-Vessiot differential Galois theory to include elliptic equations such as (\ref{equation-intro-elliptic}). Following the model-theoretic terminology, we refer to the Galois groups as {\em binding groups}. Two essential ingredients of the Galois-theoretic treatment of the algebraic relations between solutions of linear differential equations are:
\begin{itemize}
\item[(a)] the field-theoretic algebraic closure $\mathrm{acl}(-)$ is a combinatorial {\em pregeometry }, that is, it is a monotone closure operator with finite character that satisfies the {\em exchange property}: $a \in \mathrm{acl}(A,b) \setminus  \mathrm{acl}(A)  \Rightarrow b \in \mathrm{acl}(A,a)$.

\item[(b)] a Galois-theoretic characterization of algebraicity: whenever $p = \mathrm{tp}^D(f/k)$ is a type internal to the constants realized in a differential field {\em without new constants} then $p$ is an algebraic type if and only if its Galois group is trivial. \end{itemize}
In practice, these properties are often used together with a group-theoretic lemma, {\em Goursat lemma} which describes the proper algebraic subgroups of a product $G_1 \times G_2$ projecting surjectively on both factors. This leads to automatic  algebraic independence statements, see \cite{Kolchin}.

In the exponential algebraic case, the first step (a) was adapted to the study of exponential algebraicity in \cite{Kirby-semiab} where a pregeometry called {\em exponential algebraic closure}
\begin{equation}\label{equation-intro-ecl}
 \mathrm{ecl}^K(C, - ) : \mathcal P(K) \rightarrow \mathcal P(K)
\end{equation}
is defined on any differentially closed field $K$ with field of constants $C$. 
This pregeometry will be described in greater details in subsection 1.4 of the introduction. Regarding the second step (b), we prove the following theorem:

\begin{thmx}[Theorem \ref{theoremB}]  \label{intro-thmB}
Let $k$ be a differential field with an algebraically closed field of constants and $f$ an element of a differential field extension of $k$ without new constants. Assume that the algebraic closure of $k$ is {\em self-sufficient} and that $f$ satisfies a differential equation {\em internal to the constants} with parameters in $k$. Then 
\begin{quote} if $f \in \mathrm{ecl}^K(k)$ then the binding group of $\mathrm{tp}^D(f/k)$ is linear and abelian-by-finite.
\end{quote}
\end{thmx}

Here, an algebraically closed subfield $k$ extending the field $C$ of constants of some differentially closed field $K$ is {\em self-sufficient} if the analogue of the Ax-Schanuel Theorem holds over $k$ in the sense that:

\begin{quote} 
whenever $f_1,\ldots,f_n \in K$ and $e_1,\ldots,e_n \in K$ are exponentials of $f_1,\ldots, f_n$ respectively then 
$\mathrm{td}(f_1,\ldots,f_n,e_1,\ldots,e_n/k) \geq n$
provided that the $(f_i,e_i)$ are $\mathbb{Q}$-linearly independent in $(\Ga \times \Gm)(K)/(\Ga \times \Gm)(k)$.
\end{quote} 
Corollary \ref{cor-self-sufficient} shows that this condition does not depend on the choice of the differentially closed field $K$ extending $k$. 

As applications of Theorem \ref{intro-thmB}, we obtain the exponential algebraic independence of the solutions of  (\ref{equation-intro-pendulum}), as well as the exponential independence of $\mathrm{sn}(z)$,  $\mathrm{erf}(z)$ and $\mathrm{Ai}(z)$, see Corollary \ref{cor-elliptic-error-Airy}. In the case of the Bessel functions given by 
\begin{equation*}
J_\alpha(z) = \sum_{n = 0}^\infty \frac{(-1)^m}{m! \cdot \Gamma(\alpha + m + 1)} (\frac{z}{2})^{2m+ \alpha},
\end{equation*}
we get the following generalization of Kolchin's result \cite{Kolchin}.

\begin{corx}[Corollary \ref{cor-Bessel}]
Let $\alpha_1,\ldots, \alpha_n$ be complex numbers such that 
\begin{center} 
$\alpha_i \notin \mathbb{Z} + 1/2$ and $\alpha_i \pm \alpha_j \notin \mathbb{Z}$ for $i \neq j$. 
\end{center} 
Then the Bessel functions $J_{\alpha_1}(z), \ldots, J_{\alpha_n}(z)$
 and their first derivatives are $2n$ {\em exponentially algebraically independent} holomorphic functions over $\mathbb{C}(z)$. 
\end{corx}

\subsection{Relation with the model theory of the blurred complex exponential} A fundamental model-theoretic input of our framework is the first-order theory $\BE$ of the {\em blurred complex exponential} \cite{Kirby-semiab,Kirby-blurred}. We briefly describe the central role of this theory in the rest of the paper.

While Zilber's first-order treatment of complex exponential algebra is conditional to the validity of several number-theoretic conjectures, this can be avoided for the functional problems (1) and (2) following the approach of \cite{Kirby-semiab, Kirby-blurred}. Roughly speaking, it consists in ``adding a field of constants  to exponential algebra'':
\begin{quote}
 A {\em blurred exponential field} $(k,C,\Gamma)$ is a pair $k/C$  of algebraically closed fields equipped with a divisible subgroup $\Gamma \subset \Gm(k) \times \Ga(k)$   extending the trivial correspondence $\Gm(C) \times \Ga(C)$ on the field $C$ of  constants.
\end{quote}
This approach comes with a direct connection with differential algebra: every algebraically closed differential field $(k,\partial)$ carries such a structure obtained by setting 
\begin{equation}\label{equation-intro-Gamma}
 \Gamma_\mathrm{BE} := \lbrace (y,x) \in \Gm(k) \times \Ga(k) \mid \partial(y)/y = \partial(x) \rbrace.
 \end{equation}
The field $C$ is then the field of constants of $k$. A fundamental result relates the reduct of the theory $\DCF$ of ordinary differentially closed fields in the language $\mathcal L_\Gamma = \lbrace 0,1,+,-,\times, \Gamma \rbrace$ where the interpretation of $\Gamma$ is given by (\ref{equation-intro-Gamma}) to the theory deduced from  $(\mathbb{C}, \mathrm{exp})$ by ``blurring'' the graph of the exponential function. In the second case, one first constructs a countable ``field of constants'' $k_0$ by taking intersection of constants fields of {\em $E$-derivations} lying in
$$ \mathrm{EDer}(\mathbb{C}) = \lbrace D \in \mathrm{Der}(\mathbb{C}) \mid \forall a \in \mathbb{C}, D(e^{a}) = e^{a} \cdot D(a) \rbrace$$
and then interprets the correspondence $\Gamma$ as
$$ \Gamma_0 = \lbrace (y,x) \in \Gm(\mathbb{C}) \times \Ga(\mathbb{C} \mid \exists c \in k_0 \mid y = c \cdot \mathrm{exp}(x) \rbrace.$$
The main result of \cite{Kirby-blurred} states the two constructions lead to elementary substructures in the language $\mathcal L_\Gamma$. The common $\omega$-stable theory $\BE$ is called the theory of the {\em blurred complex exponential}. 

We refer the reader to Section 2 for more details on this theory and in particular for the construction of the pregeometry $\mathrm{ecl}^K(C,-)$ presented in (\ref{equation-intro-ecl})  on any model $K$ of the theory $\BE$. On the standard models $K_0 = (\mathbb{C}, k_0, \Gamma_0)$, the pregeometry  $\mathrm{ecl}^{K_0}(k_0,-)$ admits a concrete description based on the o-minimal structure $\mathbb{R}_{\mathrm{RE}}$.  As a consequence of \cite{Wilkie, Kirby-alg},  it can be phrased as an equality of two pregeometries over the field $\mathbb{C}$ of complex numbers
\begin{equation}\label{equation-intro-egalite}
\mathrm{ecl}^{K_0}(k_0, - ) = \mathrm{hcl}_{\mathbb{R}_{\mathrm{RE}}}(-) 
\end{equation}
where $\mathrm{hcl}$  is a pregeometry called {\em holomorphic closure} associated to certain reducts of the o-minimal structure $\mathbb{R}_{an}$ , see also \cite{Jones-Kirby-Servi, LGKJS}. In the case of $\mathbb{R}_{\mathrm{RE}}$,  this second pregeometry is defined by closing a subset $A$ of $\mathbb{C}$ under the evaluation of $\mathrm{RE}$-definable holomorphic functions, that is,
$$ \mathrm{hcl}(A) = \left\{ f(a_1,\ldots,a_n) \mid \begin{cases} a_1,\ldots, a_n \in A \\ f \text{ is a } k_0\text{-definable  holomorphic function} \\   (a_1,\ldots,a_n) \in \mathrm{dom}(f) \end{cases}\right\}.$$
The main result of Section 3 is a functional variant of (\ref{equation-intro-egalite})  which describes the pregeometry $\mathrm{ecl}^K(\mathbb{C},-)$ on differentially closed fields $K$ with field of constants $\mathbb{C}$.

\begin{thmx}[Theorem \ref{theorem-connection-o-minimality}]
Let $f_1,\ldots,f_n$ be holomorphic functions of one variable defined on some complex domain $U$. Then $f_1,\ldots,f_n$ are exponentially algebraically dependent over $\mathbb{C}$ {\em if and only if} there exist a subdomain $V$ of $U$, a connected open subset $W$ of $\mathbb{C}^n$ containing the analytic curve $(f_1,\ldots,f_n)(V)$ and a nonzero holomorphic function $F$ of $n$ variables definable in $\mathbb{R}_{\mathrm{RE}}$ such that 
$$F(f_1(z),\ldots, f_n(z)) = 0 \text{ for all } z \in V.$$
Furthermore, if $f_1,\ldots,f_{n-1}$ are exponentially algebraically independent over $\mathbb{C}$, the previous condition is also equivalent to the existence of a subdomain $V$ of $U$, a connected open subset $W$ of $\mathbb{C}^{n-1}$ containing the analytic curve $(f_1,\ldots,f_{n-1})(V)$ and a  holomorphic function $G$ of $n-1$ variables definable in $\mathbb{R}_{RE}$ such that
$$ f_n(z) = G(f_1(z),\ldots,f_{n-1}(z)) \text{ for all } z \in V.$$  
\end{thmx}

Here, the meromorphic functions $f_1,\ldots,f_n$ are exponentially algebraically dependent over $\mathbb{C}$ if they are exponentially algebraically independent in the sense of $\mathrm{ecl}^K(\mathbb{C},-)$ in some/any differentially closed field $K$ extending $\mathcal M(U)$ without new constants. \\

Hence, Theorem C provides {\em a differential algebraic formulation} of the complex-analytic problems (1) and (2) concerning the definability properties of holomorphic functions in the o-minimal structure $\mathbb{R}_{\mathrm{RE}}$. Replacing formulas with complete types, these problems are fully encoded in the {\em reduct map}
\begin{equation}\label{equation-intro-reduct} 
\mathrm{red}: S_n^D(\mathbb{C}(z)^{alg}) \rightarrow S_n^\BE(\mathbb{C}(z)^{alg})
\end{equation}
defined by $tp^D(a/\mathbb{C}(z)^{alg}) \mapsto tp^{\BE}(a/\mathbb{C}(z)^{alg}$ using the construction of $\BE$ as a $\DCF$-reduct. The study of the {\em stability-theoretic} description of (\ref{equation-intro-reduct}) was our initial motivation to study the problem of integration in finite terms and Theorem A and Theorem B achieve an effective description of (\ref{equation-intro-reduct}) for types which are {\em internal to the constants}. 

It is natural to expect interesting counterparts for these results at the other ends of Zilber's trichotomy, as well, such as in the description of {\em strongly minimal sets with exponentially algebraic solutions}.  Classical examples of completely disintegrated strongly minimal equations connected to Liouville's work include:

 \begin{corx}[Corollary \ref{cor-Kepler}]
Any distinct solutions $W_1(z),\ldots, W_n(z)$ of Kepler's equation and Lambert's equation respectively
\begin{equation*}
w - \mathrm{sin}(w) = z \text{ and } w \cdot e^w = z
\end{equation*} are algebraically independent over $\mathbb{C}(z)$. 
\end{corx} 
The investigation of these other stability-theoretic properties  of the reduct map (\ref{equation-intro-reduct}) is postponed until a future work. This includes important questions about the theory $\BE$ such as: are the solutions of minimal locally-modular and not geometrically trivial types {\em exponentially transcendental}? Can exponentially algebraic functions satisfy {\em strongly minimal} algebraic differential equations of arbitrary high order? Are strongly minimal systems with exponentially algebraic solutions necessarily {\em $\omega$-categorical}?

\subsection*{Acknowledgments} The authors are grateful to the University of Helsinki for its hospitality where this project started during the beginning of 2023. We are also grateful to Guy Casale, Julien Roques and Michael Singer for several comments, discussions and suggestions concerning the results of this manuscript. The first author is supported by ANR MAS 25-CE40-5294.
\section{Preliminaries on the blurred exponential reduct} 

In this section, we recall several results from \cite{Kirby-semiab} concerning exponential algebraicity in differential fields and how it is encoded in the blurred exponential reduct $\mathrm{BE}$ of the theory $\mathrm{DCF}$ of differentially closed fields of characteristic zero. Here, we only consider the exponential map of the multiplicative group $\mathrm{G}_m$ but the set up of \cite{Kirby-semiab} is more general  allowing for semiabelian varieties defined over the constants.  We refer the reader to \cite{Rosenlicht} for background on differential forms.

\subsection{Blurred exponential fields}  

\begin{defn}\label{definition-blurredexpfields}
A {\em blurred exponential field} $(k,C,\Gamma)$ is a pair $k/C$  of algebraically closed fields of characteristic zero  and equipped with a $2$-ary predicate $\Gamma$ satisfying  
\begin{itemize}
\item[(i)] $\Gamma$ is a divisible subgroup of $T\Gm(k) = \Gm(k) \times \Ga(k)$, 
\item[(ii)] $\Gamma \cap (\Gm(k) \times \Ga(C)) = \Gm(C) \times \Ga(C)$ and symmetrically $\Gamma \cap (\Gm(C) \times \Ga(k))  = \Gm(C) \times \Ga(C)$.
\end{itemize}
\end{defn} 
The second property ensures that the algebraically closed field $C$ which we call the {\em field of constants of } $(k,\Gamma)$ is definable from the $2$-ary predicate $\Gamma$ as 
$$ C = \lbrace x \in k \mid (1,x) \in \Gamma \rbrace.$$ 
For this reason, we consider blurred exponential fields as structures in the language $\mathcal L_\Gamma = \lbrace 0,1,+,-,\times, \Gamma \rbrace$. Morphisms of blurred exponential fields are morphisms of $\mathcal L_{\Gamma}$-structures.  


\begin{exam}\label{example-blurreddifferentialfields}  The correspondence $$(k,\Delta) \mapsto (k^{alg}, C(k^{alg}), \Gamma_{BE})$$
where $\Gamma_{\BE}:= \lbrace (y,x) \in T\Gm(k^{alg}) = \Gm(k^{alg}) \times \Ga(k^{alg}) \mid \frac{ \partial(y) } y = \partial(x) \text{ for all } \partial \in \Delta \rbrace$  
defines a functor for the category of differential fields with $s$ commuting derivations to the category of blurred exponential fields. 
\end{exam}

We shall focus on the case of ordinary differential fields and on models of the complete theory $\mathrm{DCF}$ of ordinary differentially closed fields of characteristic zero.

\begin{defn}\label{definition-BE}
The  {\em theory $\mathrm{BE}$ of the blurred exponential reduct} is the common and complete $\mathcal L_\Gamma$-theory of all blurred exponential fields  $(K,C,\Gamma_{\BE})$ obtained by applying the functor from Example \ref{example-blurreddifferentialfields} to ordinary differentially closed fields $(K,\partial) \models \mathrm{DCF}$.
\end{defn} 

In the language $\mathcal L_\Gamma$, an axiomatization of the theory of $(K,C,\Gamma_{\BE})$ is presented in \cite[Fact 2.3]{Kirby-blurred} as follows: the universal part of the theory  contains a countable scheme of axioms expressing the {\em Ax-Schanuel property}:

\begin{quote}
If $a \in \Gamma_{\BE}^n$ satisfies $\mathrm{td}(a/C) \leq n$ then there is a proper algebraic subgroup $J$ of $\Gm^n$ such that $a \in TJ + T\Gm(C)$.
\end{quote}
The existential part of $\BE$ expresses that $(K,C,\Gamma)$ a
is a nontrivial blurred exponential field and contains additional {\em existential closedness axioms}

\begin{quote} 
for every free and rotund subvariety $V$ of $(\Gm \times \Ga)^n$, the intersection $V \cap\Gamma^n$ is Zariski-dense in $V$.
\end{quote} 
See \cite{Kirby-blurred} for the terminology on free and rotund subvarieties. An important consequence of this explicit axiomatization is that one also obtains models of the theory $\BE$ which do not come from ordinary differential fields. 

\begin{fact}[{\cite[Theorem 1.4]{Kirby-blurred}}]\label{fact-standardmodel}
 Let $(\mathbb{C},\mathrm{exp})$ be the field of complex numbers. Consider $k_0$ the intersection of the constant fields of all the derivations in  
$$  \mathrm{EDer}(\mathbb{C}) := \lbrace D \in \mathrm{Der}(\mathbb{C}) \mid D(\mathrm{exp}(x)) = D(x) \cdot \mathrm{exp}(x) \text{ for all } x \in \mathbb{C} \rbrace $$
and the subgroup $\Gamma_0$ of $\mathbb{C}^\ast \times \mathbb{C}$ defined by:
$$\Gamma_0 := \lbrace (x,y) \in \mathbb{C}^\ast \times \mathbb{C} \mid \exists c \in k_0, y = c \cdot \mathrm{exp}(x) \rbrace $$
Then $(\mathbb{C},k_0,\Gamma_0)$ is a model of the theory $\BE$.
\end{fact}

Although it is a reduct of the theory $\DCF$, the theory $\BE$ does not share many of its desirable properties: the theory $\BE$ does not eliminate quantifiers in the language $\mathcal L_\Gamma$ (it is not even model-complete) and does not eliminate imaginaries.  To make up for the failure of these properties, we consider the continuous surjective {\em reduct  map} defined at the level of the type spaces
\begin{equation}\label{equation-reductmap}
\mathrm{red}: S^D(k) \rightarrow S^{\mathrm{BE}}(k)
\end{equation}
 coming from the realization of $\BE$ as a reduct of the theory $\DCF$. For an arbitrary reduct $\mathcal R$ of the theory $\mathrm{DCF}$ obtained by specifying a language $\mathcal L_\mathcal F = \lbrace 0,1,+,\times, \mathcal F \rbrace$  where $\mathcal F$ is a collection $\emptyset$-definable sets in $\mathrm{DCF}$, it is defined as follows: 
\begin{defn} \label{definition-reducts}
Let $k$ be a differential field. The $\mathcal F$-reduct map associated with the reduct $\mathrm{DCF} \rightarrow \mathcal R$ is the continuous surjective function $\mathrm{red} : S^\mathrm{D}(k) \rightarrow S^{\mathcal R}(k)$ 
defined by:
\begin{quote} 
if $p = \mathrm{tp}^{\mathrm{D}}(a/k)$ is any realization of $p$ in some $K \models \mathrm{DCF}$ extending $k$ then $\mathrm{red}(p) = \mathrm{tp}^{\mathcal R}(a/k)$ is the type of the same element computed in $(K,\mathcal F) \models \mathcal R$.
\end{quote}  
\end{defn}

\noindent We denote by $\mathcal L_D = \lbrace 0,1,+,\times,\partial\rbrace$ the language of differential rings.

\begin{lem} 
The reduct quotient map  $\mathrm{red} : S^\mathrm{D}(k) \rightarrow S^{\mathcal R}(k)$  is well-defined and surjective. 
\end{lem} 

\begin{proof} 
Assume that $\mathrm{tp}^{\mathrm{D}}(a/k) = \mathrm{tp}^{\mathrm{D}}(b/k)$ with $a \in K \models \mathrm{DCF}$ and $b \in L \models \mathrm{DCF}$. Since $\mathrm{DCF}$ is complete and has quantifier elimination, we can find $N \models \mathrm{DCF}$ extending $k$ such that $K \preceq_{\mathcal L_D} N$ and $L \preceq_{\mathcal L_D} N$. It follows that $K \preceq_{\mathcal L_\mathcal F} N$ and $L \preceq_{\mathcal L_\mathcal F} N$ and hence that $\mathrm{tp}^{\mathcal{R}}(a/k) = \mathrm{tp}^{\mathcal{R}}(b/k)$ as both types can be computed in $N$. The fact that it is surjective is clear.
\end{proof}

\begin{notation} \label{notation-universal}
Fix $\kappa$ a cardinal larger than the cardinal of the structures considered (e.g., one may take $\kappa = \lvert \mathbb{C} \rvert ^+$) and a saturated model $\mathcal U$ for the theory $\DCF$ of size $\kappa$ with field of constants $\mathcal C$. Hence, $(\mathcal U, \mathcal C, \Gamma_{BE})$ is also a saturated model of $\BE$.

\noindent By a substructure, we always mean a blurred exponential field $(k,C,\Gamma)$ equipped with an embedding  $i: k \rightarrow \mathcal U$ of $\mathcal L_\Gamma$-structures.
\end{notation}

It will be important to work in elementary (in the sense of model theory) submodels $K \prec_{\mathcal L_{\Gamma}} \mathcal U$ which do not extend the field of constants rather than in the universal model. The types that are realized in such models can be characterized as follows: 

\begin{lem}\label{lemma-weak-orthogonality}
Assume that $k$ is a substructure with field of constants $C$ contained in some elementary model $K$ without new constants and let  $p \in S^{\BE}(k)$. The following are equivalent: 
\begin{itemize}
\item[(i)] for some/any $a$ realizing $p$, we have $a \ind^{\BE}_k \mathcal C$.
\item[(ii)] for some/any $a$ realizing $p$, we have $\mathrm{dcl}^{\BE}(k,a) \cap \mathcal C = C,$
\item[(iii)] for some/any $a$ realizing $p$, $k(a)^{alg}$ is contained  in an elementary submodel with field of constants $C$.
 \end{itemize}
\end{lem} 
This is a general fact about $\omega$-stable theories interpreting a pure algebraically closed field $\mathcal C$. In particular, $\mathcal C$ is stably embedded and the induced structure on $\mathcal C$ eliminates imaginaries.
\begin{proof} 
(i) $\Rightarrow$ (ii) Consider $d \in \mathrm{dcl}^{\BE}(k,a) \cap \mathcal C$. Since $a \ind^{\BE}_k \mathcal C$ and $d \in \mathrm{dcl}^{\BE}(k,a)$, we have that $d \ind^{\BE}_k \mathcal C$. This means that $\mathrm{RM}^{\BE}(d/k) = \mathrm{RM}^{\BE}(d/k \cdot \mathcal C) = 0$, so that $d \in \mathrm{acl}^{\BE}(k) \cap \mathcal C \subset K \cap \mathcal C = C$.

(ii) $\Rightarrow$ (iii) Consider $L$ the prime model of $\BE$ over $k(a)$. We claim that $L$ has field of constants $C$. Indeed, if $c \in \mathcal C(L)$ then $\mathrm{tp}^{\BE}(c/k(a))$ is isolated by a formula $\phi(x)$ with parameters in $k(a)$ which (in particular) implies that $x$ is a constant. If $\phi(\mathcal U)$ is cofinite in $\mathcal C$ then it is satisfied by infinitely many elements of $C$. This contradicts that $\phi(x)$ isolates a complete type over $k(a)$ and hence $\phi(\mathcal U)$ is finite. Since $\mathcal C$ is stably embedded and eliminates imaginaries, there is an $\mathcal L$-formula $\phi(x,d)$ parametrized by its code $d \in \mathcal C$ such that $\phi(\mathcal U) = \psi(\mathcal U,d).$
By construction, we have $d \in \mathrm{dcl}^{\BE}(k,a) \cap \mathcal C = C$ and since $C$ is algebraically closed, it follows that $c \in C$ as well.

(iii) $\Rightarrow$ (i) Denote by $L$ an elementary submodel containing $k(a)^{alg}$ with field of constants $C$ and consider $c_1,\ldots,c_p \in \mathcal C$. Consider an $L$-formula $\phi(x_1,\ldots,x_p)$ of minimal Morley rank in the type $$p = tp^{\BE}(c_1,\ldots,c_p/L).$$
Since $\mathcal C$ is stably embedded, we can assume that the parameters in $\phi(x_1,\ldots,x_p)$ live in $C$ and hence in $k$. It follows that $$\mathrm{RM}^{\BE}(\phi(x_1,\ldots,x_p)) = \mathrm{RM}^{\BE}(c_1,\ldots,c_p/L) \geq \mathrm{RM}^{\BE}(c_1,\ldots,c_p/k)$$
and hence that $c_1,\ldots,c_p \ind^{\BE}_k L$. By symmetry, we have that $x \ind^{\BE}_k c_1,\ldots,c_p$ as required. 
\end{proof}

\subsection{Reformulation of the Ax-Schanuel Theorem}  Since the extensions of blurred exponential fields are algebraically closed, we say that they are {\em finite-dimensional} if they have finite transcendence degree.

\begin{defn}
Let $L/k$ be a finite-dimensional extension of blurred exponential fields with the same field of constants. The {\em exponential predimension of $L/k$} is given by:
$$ \delta(L/k) = \mathrm{td}(L/k) - \mathrm{ldim}_\mathbb{Q}(\Gamma(L)/\Gamma(k)) \in \mathbb{Z} \cup \lbrace - \infty \rbrace.$$

 Let $k$ be a substructure contained in an elementary submodel $K$ without new constants. We say that $k$ is {\em self-sufficient in $K$} if the analogue of Schanuel's conjecture holds true over $k$ in $K$, that is,
\begin{quote}
for all finite-dimensional subextensions $L/k$ with $k \subset L \subset K$, we have that $\delta(L/k) \geq 0$.
\end{quote}
\end{defn}
\noindent Corollary \ref{cor-self-sufficient} below states that this condition does not depend on the choice of the elementary submodel $K$ extending $k$ without new constants. 
\begin{exam}[Construction of self-sufficient substructures] \label{example-selfsufficient} Take $K$ to be a differentially closed field with field of constants $\mathbb{C}$ extending the field of meromorphic functions of some adequate complex domain. 

The differential fields $k = \mathbb{C}$ and $k = \mathbb{C}(z)^{alg}$ are self-sufficient in $K$. This follows from the Ax-Schanuel property stated after Definition \ref{definition-BE}. Building on these, one can construct many other examples using that: 

\begin{quote}
\textit{If $k$ is self-sufficient in $K$, any finite-dimensional extension of blurred exponential field $L/k$ in $K$ satisfying $\delta(L/k) = 0$ is self-sufficient in $K$.}
\end{quote} 
\begin{proof}
$\delta(M/L) = \delta(M/k) - \delta(M/L) = \delta(M/k) \geq 0$ by self-sufficiency of $k$.
\end{proof} 
In particular, every elementary extension of $\mathbb{C}(z)$ is self-sufficient in $K$.
Other examples where this lemma applies include the differential field $ \mathbb{C}(z,W(z))^{alg}$ generated by the Lambert function. A second process that we will use to construct self-sufficient substructures is: 
\begin{quote} 
\textit{any extension $L/k$ of blurred exponential field  obtained by adjoining an exponentially transcendental element in $K$ (see below) to a self-sufficient substructure $k$ in $K$ is also self-sufficient}. 
\end{quote}
Using this, we will prove that the differential fields  $ \mathbb{C}(z, \mathrm{sn}(z))^{alg}$ and  $\mathbb{C}(z)\langle \mathrm{erf}(z) \rangle^{alg}$ are self-sufficient. Similarly, the differential field $k \langle y \rangle^{alg}$ where $y \in K$ is differentially transcendental over $k$ and $k$ is self-sufficient is also self-sufficient in $K$.

On the other hand, an obvious example where self-sufficiency fails is the substructure $k = \mathbb{C}(z,e^{e^{z}})^{alg}$ as if $L = k(e^z)^{alg}$, we have $\delta(L/k) = - 1$. Note that however in this example the substructure is not a {\em differential field}. An example of a {\em non-self-sufficient differential field} referred as ``Singer's example'' in \cite[Remark 1]{Rosenlicht-Singer} is 
$$k = \mathbb{C}(z, \mathrm{ln}(z) + \alpha \mathrm{ln}(z+1))^{alg}$$
where $\alpha \in \mathbb{C}$ is any irrational number.
\end{exam}

Proposition \ref{proposition-predimension} below which generalizes Proposition 3.7 from \cite{Kirby-semiab} and provides a geometric characterization of this predimension $\delta(-/k)$ over a self-sufficient differential field $k$.

\begin{lem} \label{lemma-theta-definition}
Let $L/k$  be an extension of algebraically closed field of characteristic zero. The function $\theta : ( \mathrm{G}_m \times \mathrm{G}_a)(L) /(\mathrm{G}_m \times \mathrm{G}_a)(k) \rightarrow \Omega^1(L/k)$ defined by $$\theta(y,x) = dy/y - dx$$is an injective group homomorphism.
\end{lem}

\begin{proof} 
The fact that it is a group morphism is clear. To see that it is injective, note that if $(y,x) \in (\mathrm{G}_m \times \mathrm{G}_a)(L)$ satisfies $dy/y - dx = 0 \in \Omega^1(L/k)$ then  \citep[Proposition 4]{Rosenlicht} implies that $dy = dx = 0 \in \Omega^1(L/k)$ so that $x,y \in k^{alg} = k$.
\end{proof}

\begin{defn}
Let $L/k$ be an extension of blurred exponential fields. The image of the $\mathbb{Q}$-vector space $\Gamma(L)/\Gamma(k)$ in $\Omega^1(L/k)$ under the group homomorphism $\theta$ given by Lemma \ref{lemma-theta-definition} is denoted $\Theta_{\BE}(L/k)$ and called the space of {\em exponential forms} on $L/k$.
\end{defn}

\begin{Prop} \label{proposition-predimension}
Assume that $k$ is a {\em self-sufficient} differential field embedded in a differentially closed field $K$ without new constants. 

\begin{quote} If $\theta_1,\ldots, \theta_n \in \Theta_{\BE}(K/k)$ are $\mathbb{Q}$-linearly independent then they remain $K$-linearly independent in $\Omega^1(K/k)$.
\end{quote}
\end{Prop}

Combining  Proposition \ref{proposition-predimension} with the previous lemma, we can compute the exponential predimension of any finite-dimensional blurred exponential field extension $L/k$ in $K$ as follows: 
$$ \delta(L/k) = \mathrm{ldim}_L(\Omega^1(L/k)) - \mathrm{ldim}_L (L \ \cdot \Theta_{\BE}(L/k)) = \mathrm{ldim}_L(\mathrm{Der}_{\BE}(L/k))$$
where $\mathrm{Der}_{\BE}(L/k)$ denotes the annihilator of $\Theta_{\BE}(L/k)$ in the perfect pairing 
$$ \Omega^1(L/k) \times \mathrm{Der}(L/k) \rightarrow L $$
In particular, $\delta(L/k) \geq 0$, hence the hypothesis that $k$ is self-sufficient in $K$ \textit{can not be avoided} in the assumptions of Proposition \ref{proposition-predimension}. In fact, this argument shows that an algebraically closed differential field satisfies the conclusion of Proposition \ref{proposition-predimension} if and only if it is self-sufficient in $K$. See also Lemma \ref{lemma-exponentialalgebra}.

\begin{proof}[Proof of Proposition \ref{proposition-predimension}] The classical statement which goes back to \cite{Ax-Theorem, Rosenlicht} is that if  $\theta_1,\ldots, \theta_n$ are $C$-linearly independent then they are $K$-linearly independent. Here, we assert that one can even descend to $\mathbb{Q}$ assuming that $k$ is self-sufficient in $K$. 

This is done as follows: Otherwise, by contradiction, in $\Omega^1(K/k)$, we can find $(x_i,y_i) \in \Gamma_{\BE}(K)$ such that the $\theta_i = dy_i/y_i - dx_i$ are $\mathbb{Q}$-linearly independent but satisfy a nontrivial $C$-linear relation of the form: 
$$ c_1 \cdot \theta_1 + \ldots + c_n \cdot \theta_n = 0 \in \Omega^1(K/k).$$
Assuming that this relation is minimal, we have furthermore that all the $c_i$ are $\mathbb{Q}$-linearly independent. Set $L = k(x_i,y_i, i = 1,\ldots, n)^{alg}$. We claim that that $\delta(L/k) < 0$ which contradicts the assumption that $k$ is self-sufficient in $K$.

Indeed, on the one hand, \cite[Proposition 4]{Rosenlicht} implies that the previous equality can be rewritten as
$$ d(c_1 \cdot x_1 + \ldots + c_n \cdot x_n) = \frac{dy_1}{y_1} = \cdots = \frac{dy_n}{y_n} = 0 \in \Omega^1(K/k).$$
This means that $y_1,\ldots,y_n, c_1 \cdot x_1 + \ldots + c_n \cdot x_n \in k$ and hence that $\mathrm{td}(L/k) \leq n - 1$. On the other hand, the assumption that $\theta_1,\ldots,\theta_n$ are $\mathbb{Q}$-linearly independent implies that $\mathrm{ldim}_\mathbb{Q}(\Gamma_{\BE}(L)/\Gamma_{\BE}(k)) \geq n$. The two together imply that $\delta(L/k) < 0$. This contradicts that $k$ is self-sufficient in $K$ and concludes the proof.
\end{proof}

The following lemma which will be used in Section 3 has a similar flavor. The differential assumptions from Proposition \ref{proposition-predimension} are removed but the conclusion is also weaker.

\begin{lem}\label{lemma-exponentialalgebra} 
Let $L/k$ be an extension of blurred exponential fields with the same field of constants. The space $\mathrm{Der}_{\BE}(L/k)$ of $\BE$-derivations is an $L$-vector space. A lower bound for its dimension is given by:
$$ \delta(L/k) \leq \mathrm{ldim}_L(\mathrm{Der}_{\BE}(L/k)).$$
Furthermore, if $k$ is self-sufficient in $L$ in the sense that $\delta(M/k) \geq 0$ for all intermediate extension contained in $L$ then 
$$ \delta(L/k) = 0 \text{ if and only if } \mathrm{Der}_{\BE}(L/k) = \lbrace 0 \rbrace.$$
\end{lem}

\begin{proof} 
For the first part, we have that
$\mathrm{ldim}_Q(\Theta_{\BE}(L/k)) \geq \mathrm{ldim}_L(L\cdot\Theta_{\BE}(L/k))$
so that 
$$\delta(L/k) \leq \mathrm{ldim}_L( \Omega^1(L/k)) -  \mathrm{ldim}_L(L\cdot\Theta_{\BE}(L/k)).$$
Since $\mathrm{Der}_{\BE}(L/k)$ is the annihilator of $\Theta_{\BE}(L/k)$ in $\Omega^1(L/k)$, the first part follows. 

To prove the second part, consider a nontrivial $\BE$-derivation $D \in \mathrm{Der}_{\BE}(L/k)$. Denote by $M \subsetneq L$ its field of constants. The Ax-Schanuel Theorem applied to the differential field $(L,D)$ implies that $\delta(L/M) \geq 1$ as $M \neq L$ and self-sufficiency of $k$ in $L$ implies that $\delta(M/k) \geq 0$. From additivity of the predimension, we conclude 
$$ \delta(L/k) = \delta(L/M) + \delta(M/k) \geq 1.$$
This implies that $\delta(L/k) \neq 0$. The other direction follows from the first part.
\end{proof} 

\subsection{Exponential hulls} Fix $k$ a self-sufficient substructure contained in a model $K \models \BE$ without new constants.

\begin{lem}[{\citep[Lemma 3.9]{Kirby-semiab}}] \label{lemma-submodularity} The exponential predimension is {\em submodular },  that is,
$$\delta(M \cap N /k) + \delta((M \cdot N)^{alg}/k) \leq \delta(M/k) + \delta(N/k)$$
for all finite-dimensional blurred exponential field subextensions $M/k$ and $N/k$ in $K$ of finite transcendence degree.
\end{lem}
Importantly for the validity of this lemma, all the extensions of blurred exponential fields are assumed to be algebraically closed. Note also that taking 
$$ M = \mathbb{Q}(z + \pi)^{alg} \text{ and } N = \mathbb{Q}(z)^{alg}$$
then both $M/k$ and $N/k$ have no new constants but $(M \cdot N)^{alg}/k$ has $\pi$ as a new constant, hence working inside an elementary submodel $K$ without new constants is required.

Let $L/k$ be a  finite-dimensional extension of substructures in $K$. As a consequence of Lemma \ref{lemma-submodularity}, the set of finite-dimensional extensions $M/k$ in $K$ satisfying that:

\begin{quote}
$M$ extends $L$ and $\delta(M/k)$ is minimal among the finite-dimensional extensions of $L$ in $K$,
\end{quote}
is stable under finite intersections and algebraic closure of the compositum.

 \begin{defn} 
Let $L/k$ be a finite-dimensional extension of substructures in $K$. The {\em hull of $L/k$} in $K$ denoted $\mathrm{Hull}_K(L)$ is the smallest (necessarily finite-dimensional) extension of $k$ in $K$ containing $L$ with minimal predimension over $k$. When $a = a_1,\ldots, a_n \in K$, we also write $\mathrm{Hull}_K(a/k)$ for $\mathrm{Hull}_K(k(a)^{alg})$.
\end{defn}

The following lemma shows that $\mathrm{Hull}_K(L)$ does not depend on $k$. More generally, for a (not necessarily finite-dimensional) substructure $L$ extending the field $C$ of constants, one can define its hull as the smallest self-sufficient extension of $L$ in $K$. 

\begin{lem} \label{lemma-self-sufficient}
Let $L/k$ be a finite-dimensional extension in $K$. Then $\mathrm{Hull}_K(L)$ is the smallest self-sufficient extension of $L$ in $K$. 

In particular, $L$ is self-sufficient in $K$ if and only if for all tuple $a$ from $L$, we have $\mathrm{Hull}(a/k) \subset L$.
\end{lem}

\begin{proof} 
Certainly, $N = \mathrm{Hull}_K(L)$ is self-sufficient in $K$ since $\delta(M/N) \leq 0$ implies that $$\delta(M/k) = \delta(M/N) + \delta(N/k) \leq \delta(N/k).$$
Hence, by minimality of $\delta(N/k)$, we must have $\delta(M/N) = 0$. Conversely, consider $M/L$ an algebraically closed extension of $L$ self-sufficient in $K$. By submodularity, we have:
$$\delta(M \cap N /k)  \leq \delta(M/k) + \delta(N/k) - \delta(M \cdot N/k) = \delta(N/k) - \delta(M \cdot N/M)$$
Since $M$ is self-sufficient in $K$, the latter is bounded from above by $\delta(N/k)$. By minimality of $\delta(N/k)$, we must that $\delta(M\cap N/k) = \delta(N/k)$ and by minimality of $N$, we have $N \subset M \cap N$ as required.

The second part follows from the fact that an increasing union of self-sufficient substructures is self-sufficient in $K$.
\end{proof}

{\exam Take $K$ to be a differentially  closed field extending $\mathcal M(U)$ for some well-chosen $U$ as previously and $k = \mathbb{C}(z)^{alg}$. If a meromorphic function $f$ is exponentially transcendental (see below) then $\mathrm{Hull}(f/k) = k(f)^{alg}$. On the other hand, for exponential algebraic functions, the exponential hull can be larger. For instance,
 $$\mathrm{Hull}_K( \mathrm{ln}(z) + \alpha \mathrm{ln}(z+1)/\mathbb{C}(z)^{alg}) = \begin{cases} \mathbb{C}(z, \mathrm{ln}(z),\mathrm{ln}(z+1))^{alg} \text{ if } \alpha \notin \mathbb{Q} \\ \mathbb{C}(z,\mathrm{ln}(z^q/(z+1)^p)^{alg} \text{ if } \alpha = p/q \end{cases} $$
Similarly, the weighted sums of logarithms appearing in Liouville's theorem \cite{Liouville2}
$$\mathrm{ln}(g_1(z)) + c_2 \cdot \mathrm{ln}(g_2(z)) + \ldots + c_n \cdot \mathrm{ln}(g_n(z)) $$
where $c_1 = 1,\ldots, c_n \in C$ are $\mathbb{Q}$-linearly independent, $g_1,\ldots,g_n$ multiplicatively independent modulo $\mathbb{C}$ give examples where the hull can be arbitrary large. Note that if $f$ is an elementary function (in the sense of Liouville) then $\mathrm{Hull}_K(f/k)$ is the algebraic closure of any minimal elementary differential field extension of $k$ in $K$ containing $f$. It is in general larger than the algebraic closure of the differential field generated by $f$ over $k$.}

\begin{fact} \label{fact-caracterizationoftypes}
Let $k$ be a substructure and let $K_i$ be elementary submodels extending $k$ without new constants such that $k$ is self-sufficient in both $K_1$ and $K_2$. 
Consider $a_i$ be tuple from $K_i$ for $i = 1,2$. Then
$$tp^{\BE}(a_1/k) = tp^{\BE}(a_2/k)$$
if and only if  there exists an isomorphism of $\mathcal L_\Gamma$-structures 
$$\phi: \mathrm{Hull}_{K_1}(a_1/k) \rightarrow \mathrm{Hull}_{K_2}(a_2/k)$$
fixing $k$ and sending $a_1$ to $a_2$.
\end{fact}

The fact follows from the  construction of the theory $\BE$ as the theory of the Hrushovski construction associated to the exponential predimension, see \citep[Theorem 2.22]{Kirby-semiab}.

\begin{cor} \label{cor-self-sufficient}
Let $k$ be a substructure and and let $K_i$ be elementary submodels extending $k$ without new constants. Then $k$ is self-sufficient in $K_1$ if and only if $k$ is self-sufficient in $K_2$.
\end{cor} 

\begin{proof} 
Consider $K_3$ the prime model of $\BE$ which elementary embeds in both $K_1$ and $K_2$. It is enough to show that $k$ is self-sufficient in $K_1$ if and only if $k$ is self-sufficient $K_3$, so we may assume that $K_1 \prec K_2$ to prove the corollary.

The Ax-Schanuel property ensures that $C$ is self-sufficient in both $K_1$ and $K_2$. Take $a$ be a tuple from $k$. Since moreover
$$tp^{K_1}(a/C) = tp^{K_2}(a/C)$$
Fact \ref{fact-caracterizationoftypes} applies and we conclude that
$$ \delta(\mathrm{Hull}_{K_1}(a/C)/C) =  \delta(\mathrm{Hull}_{K_2}(a/C)/C).$$
Hence, $\mathrm{Hull}_{K_1}(a/C)$ is an extension of $C(a)^{alg}/C$ of minimal predimension in $K_2$ and it follows that $\mathrm{Hull}_{K_2}(a/C) \subset \mathrm{Hull}_{K_1}(a/C)$ by definition of the hull. Since Fact \ref{fact-caracterizationoftypes} also implies that 
$$\mathrm{td}(\mathrm{Hull}_{K_1}(a/C)/C) =  \mathrm{td}(\mathrm{Hull}_{K_2}(a/C)/C) $$
and both are finite-dimensional and algebraically closed, it follows that they are equal. We conclude by Lemma \ref{lemma-self-sufficient} that $k$ is self-sufficient in $K_1$ if and only if $k$ is self-sufficient in $K_2$.
\end{proof}

\subsection{Exponential algebraicity as a pregeometry} Let $K \models \BE$ be a model with field of constants $C$. Recall that a combinatorial pregeometry on $K$ is a closure operator
$$ \mathrm{cl}: \mathcal P(K) \rightarrow \mathcal P(K) $$
satisfying monotonicity, finite character and the exchange property: 
\begin{center}
$a \in \mathrm{cl}(Ab) \setminus \mathrm{cl}(B)$ if and only if $b \in \mathrm{cl}(Ab) \setminus \mathrm{cl}(A)$. 
\end{center} See for example \cite[Chapter 2, Section 1]{Pillay} for  the basic properties of pregeometries. 

\begin{exam}
Assume furthermore that $K \models \DCF$. The closure operator $\mathrm{ELA}$ on $K$ defined by: 
\begin{quote}
$b \in \mathrm{ELA}(C, A)$ if $b$ belongs to an elementary extension of the differential field $C \langle A \rangle$ generated by $A$,
\end{quote}
is {\em  not} a pregeometry. It satisfies monotonicity, finite character but not the exchange property. Indeed, take $a = z$ and $b = W(z)$ the Lambert $W$-function solution of $w \cdot e^w = z$. It follows from Liouville's result asserting that $W(z)$ is not an elementary function that $b \notin \mathrm{ELA}(C, a)$. But the functional equation above shows that $a \in \mathrm{ELA}(C,b)$.
\end{exam}

\begin{defn} 
Let $a$ be a tuple from $K$. The {\em exponential transcendence degree} of $a$ over $C$ is defined by
$$ \mathrm{etd}^K(a/C) = \delta(\mathrm{Hull}(C(a)^{alg}/C)) \geq 0.$$
Let $A$ be a subset of $K$. We say that $b$ is {\em exponentially algebraic} over $A,C$  and write $b \in \mathrm{ecl}^K(C,A)$ if for some finite tuple $a$ from $A$, we have  
$$ \mathrm{etd}^K(a,b/C) = \mathrm{etd}^K(a/C).$$ 
We also write $\mathrm{ecl}^K(k)$ for $\mathrm{ecl}^K(C,k)$ when $k$ is a substructure containing the field of constants $C$ of $K$.
\end{defn}

\begin{fact}[{\citep[Section 2.7]{Kirby-semiab}}]\label{fact-pregeometry}
Let $K$ be a model of $\BE$ with field of constants $C$. The closure operator
$$\mathrm{ecl}^K(C, -) : \mathcal P(K) \rightarrow \mathcal P(K)$$ is a pregeometry. Moreover, whenever $k$ is an algebraically closed substructure of $K$
\begin{quote}
the (type-definable) condition ``$y \notin \mathrm{ecl}^K(k)$'' isolates the unique complete one-type $p_\omega \in S^{\BE}(k)$ of rank $\omega$.
\end{quote}
\end{fact} 

The complete type $p_\omega \in S^{\BE}(k)$ is called the {\em exponentially transcendental type} over $k$. It is the image of the differentially transcendental type over  $k$ in $\DCF$ under the reduct map (\ref{equation-reductmap}). It is also the image of many other types under the reduct map. For instance, the meromorphic functions
$$ \mathrm{erf}(z) = \int e^{-z^2}dz, \mathrm{sn}(z) =  \int \frac{dz}{z^3 + az + b}$$
which will be shown to be exponentially transcendental all satisfy the type $p_\omega \in S^{\BE}(\mathbb{C}(z))$ in the blurred exponential reduct $\BE$. They of course all have different differential types. 

We summarize below the properties of this pregeometry that will be used in the rest of the text. 
\begin{cor} \label{corollary-predimension}
Let $K \models \BE$ with field of constants $C$ and let $k$ be a substructure extending $C$ and self-sufficient in $K$.
\begin{itemize} 
\item[(i)] If $a$ is a tuple from $K$ then $a \in \mathrm{ecl}^K(k)$ if and only if $a$ is contained in a subtructure $N$ of $K$ satisfying $\delta(N/k) = 0$ if and only if $\delta(\mathrm{Hull}_K(a/k)/k) = 0$.
\item[(ii)] An element $a \in K$ is exponentially transcendental over $k$ if and only if
\begin{center} 
$\delta(k(a)^{alg}/k) = 1$ and $k(a)^{alg}$ is self-sufficient.
\end{center}
\item[(iii)] If $a$ is a tuple from $K$ and $a \in \mathrm{ecl}^{K}(k)$ then for any realization $b$ of $tp^{\BE}(a/k)$ and any elementary submodel $L$ containing $b$ with field of constants $C$, we have that $a \in \mathrm{ecl}^{L}(k)$.
\item[(iv)] A tuple $a$ from $K$ is exponentially algebraic over $k$ if and only if $tp^{\BE}(a/k)$ has finite Morley rank.
\end{itemize} 
\end{cor} 

\begin{proof}
(ii) follows from (i), (iii) follows from (i) together with Fact \ref{fact-caracterizationoftypes} and (iv) reduces to the case of one-types using the existence of bases for the pregeometry $\mathrm{ecl}^K(C,-)$ and hence follows from Fact \ref{fact-pregeometry}.  

We now show the first equivalence in (i) as the second follows directly from the definition. Assume first that $a \in \mathrm{ecl}^K(k)$. This means that we can find a finite tuple $b$ in $k$ such that $a \in \mathrm{ecl}^K(C,b)$. Setting $k_1 = \mathrm{Hull}_K(b/k)$, this means that $a$ is contained in a blurred exponential field extension $N$ of $k_1$ satisfying $\delta(N/k_1) = 0$.
By submodularity of the predimension function, for any finite-dimensional self-sufficient extension $k_1 \subset k_2 \subset K$, we have that 
$$\delta((k_2 \cdot N)^{alg}/C) \leq \delta(k_2/C) + \delta(k_1/C) - \delta(k_1 \cap k_2/C) = \delta(k_2/C).$$  
It follows that $\delta((k_2 \cdot N)^{alg}/k_2) = 0$ since $k_2$ is self-sufficient in $K$. Using the finite character of the predimension $\delta$, we obtain that $\delta((k \cdot N)^{alg}/k) = 0$. Hence, $(k \cdot N)^{alg}/k$ is an extension with predimension zero containing $a$. 

For the converse, assume now that $a$ is contained in $N$ such that $\delta(N/k) = 0$. Take $x_1,\ldots,x_n \in \Gamma(N)$ generators of $\Gamma(N)/\Gamma(k)$ where $n = \mathrm{td}(N/k)$ so that $N = k(x_1,\ldots,x_n)^{alg}$. Set $k_0$ to be the field of definition for the ideal 
$$ I(a,x_1,\ldots,x_n)  = \lbrace f \in k[X_0,X_1,\ldots,X_{2n}] \mid f(a,x_1,\ldots, x_n) = 0 \rbrace$$
and $k_1 = \mathrm{Hull}_K(k_0)$. By construction, we have that $a \in k_1(x_1,\ldots,x_n)^{alg}$ and $$\mathrm{td}(k_1(x_1,\ldots,x_n)^{alg}/k_1) = \mathrm{td}(k(x_1,\ldots,x_n)^{alg}/k).$$
It follows that $\delta(k_1(x_1,\ldots,x_n)^{alg}/k_1) \leq 0$ and hence equal to zero since $k_1$ is self-sufficient in $K$. Since $k_1$ is finitely generated say by some $b \in K$, we have shown that $a \in \mathrm{ecl}^K(C,b)$ as required.
\end{proof}

\section{Relation with the o-minimal structure $\mathbb{R}_{\mathrm{RE}}$}

An important feature of this pregeometry that we used to formulate our results in the introduction is its concrete description based on the o-minimal structure $$\mathbb{R}_{\mathrm{RE}} := (\mathbb{R}, 0,1,+,\times, < , \text{restricted complex exp})$$
described in \cite{Kirby-alg} for the standard model $(\mathbb{C}, k_0, \Gamma_0)$ of $\BE$ of \ref{fact-standardmodel}. To formulate a similar description in the differential context, we work with holomorphic functions $f_1,\ldots, f_n$ of one variable defined on some complex domain $U$ (i.e. a connected open subset of the complex plane). Fixing $K \models \DCF$ extending the differential field $\mathcal M(U)$ of meromorphic functions on $U$ without new constants, we say that $f_1,\ldots,f_n$ are {\em exponentially algebraic independent over $\mathbb{C}$} if they are independent for the pregeometry $\mathrm{ecl}^{K}(\mathbb{C},-)$ defined on $K$ in the previous section. This property does not depend on the choice of the differentially closed field $K$ extending $\mathcal M(U)$, see Corollary \ref{corollary-predimension}. 

The main theorem of this section describes this pregeometry as follows:

\begin{theorem}[Theorem C]\label{theorem-connection-o-minimality}
Let $f_1,\ldots,f_n$ be holomorphic functions of one variable defined on some complex domain $U$. Then $f_1,\ldots,f_n$ are exponentially algebraically dependent over $\mathbb{C}$ {\em if and only if} there exist a subdomain $V$ of $U$, a connected open subset $W$ of $\mathbb{C}^n$ containing the analytic curve $(f_1,\ldots,f_n)(V)$ and a nonzero holomorphic function $F$ of $n$ variables definable in $\mathbb{R}_{RE}$ such that 
$$F(f_1(z),\ldots, f_n(z)) = 0 \text{ for all } z \in V.$$
Furthermore, if $f_1,\ldots,f_{n-1}$ are exponentially algebraically independent over $\mathbb{C}$, the previous condition is also equivalent to the existence of a subdomain $V$ of $U$, a connected open subset $W$ of $\mathbb{C}^{n-1}$ containing the analytic curve $(f_1,\ldots,f_{n-1})(V)$ and a  holomorphic function $G$ of $n-1$ variables definable in $\mathbb{R}_{RE}$ such that
$$ f_n(z) = G(f_1(z),\ldots,f_{n-1}(z)) \text{ for all } z \in V.$$  
\end{theorem}

\subsection{Wilkie's pregeometry}  We recall here Wilkie's results specialized to the o-minimal structure $\mathbb{R}_{\mathrm{RE}}$. A holomorphic function (of several variables) is said to be definable if it is definable with parameters in this structure. We say more precisely that it is $k_0$-definable to mean that the parameters are taken inside a subfield $k_0$ of $\mathbb{C}$. For example, the restrictions to open boxes of the holomorphic functions in the (differential) $\mathbb{C}$-algebras  $$\mathcal A = ( \mathbb{C}[z_1,\ldots,z_n,e^{z_1},\ldots, e^{z_n}], n \geq 0 )$$
are all definable. Fact \ref{fact-Wilkie-complete} below describes the germs of all definable functions at ``generic points'' using these $\mathbb{C}$-algebras. {\em Fix $k_0$ a countable algebraically closed subfield of $\mathbb{C}$ and expand the o-minimal structure $\mathbb{R}_\mathrm{RE}$ by constant symbols for the elements of $k_0$.}

\begin{defn} 
let $A$ be a subset of $\mathbb{C}$. The {\em holomorphic closure } of $A$ is the set of values of $k_0$-definable holomorphic functions evaluated at points in $A$, that is, 
$$ \mathrm{hcl}(A) = \lbrace f(a_1,\ldots,a_n) \mid a_1,\ldots,a_n \in A, f \text{ is } k_0\text{-definable and } a \in \mathrm{dom}(f) \rbrace.$$
We also write $\mathrm{hcl}_{k_0}(A)$ when it is desirable to stress more precisely the countable algebraically closed subfield $k_0$. This coincides with the notion of holomorphic closure from \cite{LGKJS, Jones-Kirby-Servi} for the o-minimal structure $\mathbb{R}_{\mathrm{RE}}$ where the elements of $k_0$ have been named by constant symbols.
\end{defn} 
Because $k_0$ is countable, the collection of $k_0$-definable function is countable. Hence if $A$ is finite or countable, then  $\mathrm{hcl}(A)$ is a countable algebraically closed subfield extending $k_0(A)^{alg}$.
\begin{fact} \label{fact-Wilkie-complete}
Let $k_0$ be a countable algebraically closed subfield of parameters.
The following properties hold:
\begin{itemize} 
\item[(a)] $\mathrm{hcl}_{k_0}(-): \mathcal P(\mathbb{C}) \rightarrow \mathcal P(\mathbb{C})$ is a pregeometry on $\mathbb{C}$. Furthermore, $a_1,\ldots,a_n \in \mathbb{C}$ are $\mathrm{hcl}$-independent if and only if for any $k_0$-definable holomorphic function $F$ with $(a_1,\ldots,a_n) \in \mathrm{dom}(F)$ then 
$$ F(a_1,\ldots,a_n) = 0 \Rightarrow F = 0.$$

\item[(b)] If $g$ is a $k_0$-definable holomorphic function of the variables $z = (z_1,\ldots, z_n)$ and $ b = (b_1,\ldots,b_n) \in \mathrm{dom}(g)$ are $\mathrm{hcl}$-independent then $g$ can be $m$-implicitly defined from $\mathcal A$ in a neighborhood $W$ of $b$ in the following sense: 

\begin{quote} 
there is $m \geq 1$, holomorphic functions $g_1 = g, g_2,\ldots,g_m$ on $W$ and $F_1,\ldots,F_m \in \mathcal A_{n + m} = \mathbb{C}[z,e^{z},w,e^{w}]$ such that for all $z \in W$ 
$$ F_i\Big(z,e^{z}, g_1(z), \ldots, g_m(z), e^{g_1(z)}, \ldots, e^{g_m(z)} \Big) = 0 $$
for $i = 1,\ldots, m$ and $\mathrm{det}\Big((\frac{\partial F_i} {\partial w_j})_{i,j \leq m}(\overline{z}, g_1(\overline{z}),\ldots, g_m(\overline{z})) \Big)\neq 0$.
\end{quote}
\end{itemize}
\end{fact}

\begin{proof} 
The first property is proved (in greater generality) in \cite[Subsection 2.2]{Wilkie}. Similarly, in an o-minimal structure $\mathcal R$ generated by a class $\mathcal A$ of holomorphic functions, the main result of \cite{Wilkie} states that the germ of a definable function around an $\mathrm{hcl}$-generic tuple $(b_1,\ldots,b_n)$  can be obtained using only finitely many applications of composition, Schwarz reflection, taking partial derivatives and extracting implicit functions. The refinement formulated in (b) for the o-minimal structure $\mathbb{R}_{\mathrm{RE}}$ is obtained in \cite{LGKJS} by showing that the class of holomorphic functions obtained from such Khovanskii systems is stable under all these operations, see \cite[Lemma 4.3]{LGKJS} and its proof. 
\end{proof} 

\subsection{$\mathrm{hcl}$-independent holomorphic functions} 
Based on Fact \ref{fact-Wilkie-complete} (a), we introduce the following definition for holomorphic functions of one variable.

\begin{defn} 
Holomorphic functions $f_1,\ldots,f_n$ defined on the same complex domain $U$ are {\em $\mathrm{hcl}$-independent over $\mathbb{C}$} if for any subdomain $V$ of $U$, any connected open subset $W$ of $\mathbb{C}^n$  such that $(f_1,\ldots,f_n)(V) \subset W$ and any $F$ definable holomorphic function on $W$, we have:
$$ F \circ (f_{1\mid V}, \ldots, f_{n \mid V}) = 0 \Rightarrow F = 0.$$
\end{defn}

The proof of the following lemma is analogous to Proposition 3.3 from \cite{Jones-Kirby-Servi}.

\begin{lem} \label{lemma-generic-points}

Let $f_1,\ldots,f_n$ be holomorphic functions defined on some complex domain $U$ which are $\mathrm{hcl}$-independent over $\mathbb{C}$ and let $k_0$ be a countable subfield of $\mathbb{C}$. Then the analytic curve $(f_1,\ldots,f_n)(U)$ passes through an $\mathrm{hcl}$-generic point, that is, there exists $a \in U$ such that the values
$f_1(a),\ldots,f_n(a)$ are $\mathrm{hcl}$-independent.
\end{lem} 

\begin{proof} 
Consider the countable set $\mathcal S$ of all the pairs $(W,F)$ where $W \subset \mathbb{C}^n$ is box with coordinates in $\mathbb{Q}^{alg}$ meeting nontrivially the analytic curve  $(f_1,\ldots,f_n)(U)$ and $F$ is an $k_0$-definable holomorphic function on $W$ which is not identically zero.

Without loss of generality, we may assume up to restricting $U$ that $f_1,\ldots,f_n$ are defined in some o-minimal structure extending semialgebraic geometry. It follows that for any pair $(W,F) \in \mathcal S$
\begin{itemize} 
\item the open subset $\lbrace a \in U \mid (f_1(a),\ldots, f_n(a)) \in W \rbrace$ has finitely many connected components $V_1,\ldots V_p$,
\item on each $V_j$, by $\mathrm{hcl}$-independence of the functions, we have that $F \circ (f_1,\ldots,f_n) \neq 0$. As it is an analytic function, the set of points $a \in V_j$ such that this function vanishes is countable. 
\end{itemize}
It follows that 
$$\mathcal E_{W,F} := \lbrace a \in U \mid f_1(a),\ldots, f_n(a) \in \mathrm{dom}(F) \text{ and } F(f_1(a),\ldots,f_n(a)) = 0 \rbrace$$
is countable and so is the union over all pairs $(W,F) \in \mathcal S$. Taking $b$ outside of this countable set, we conclude that $b$ is $\mathrm{hcl}$-generic.
\end{proof}

The following argument is extracted from Wilkie's proof of Fact \ref{fact-Wilkie-complete} (a). 

\begin{Prop} \label{prop-from relation to functions}
Let $f_1,\ldots,f_n,f$ be holomorphic functions. Assume that $f_1,\ldots,f_n$ are $\mathrm{hcl}$-independent over $\mathbb{C}$ and that $f_1,\ldots,f_n,f$ are not. Then there exists a subdomain $V$ of $U$, $W$ a connected open subset of $\mathbb{C}^n$ and a definable holomorphic function $G$ on $W$ such that
$$ f_{\mid V} = G \circ (f_{1 \mid V},\ldots,f_{n \mid V}) $$
\end{Prop}

\begin{proof} 
Take $F$ a nontrivial definable relation between $f_1,\ldots,f_n,f$.
We first claim that for some $i$
$$  \frac{\partial^i F} {\partial Y_{n+1}^i}(f_1,\ldots,f_n,f) \neq 0$$
Assume otherwise and take by the previous lemma a point $a$ such that $b_1 = f_1(a),\ldots, b_n = f_n(a)$ are exponentially algebraically independent over a countable field $k_0$ of definition of $F$. Then we have 
$$ \frac{\partial^i F} {\partial Y_{n+1}^i}(b_1,\ldots,b_n,b) = 0 \text{ for all } i$$
where $b = f(a)$. It follows that the analytic function 
$$ z \mapsto F(b_1,\ldots,b_n,z) $$
vanishes identically in a neighborhood $U_b$ of $b$. Since $\mathrm{hcl}$-generic elements are dense in $\mathbb{C}$, we can find $b_{n+1}$ in $U_b$ which is $\mathrm{hcl}$-independent with $b_1,\ldots,b_n$ so that $F(b_1,\ldots,b_{n+1}) = 0$. This contradicts Fact \ref{fact-Wilkie-complete} (a) and finishes the proof of the claim.

Now take $i$ minimal with the property that 
$$  \frac{\partial^i F} {\partial Y_{n+1}^i}(f_1,\ldots,f_n,f) \neq 0.$$
Since $F \neq 0$, at least one of the two functions 
$$ F_\pm = \pm F + \frac{\partial^{i-1} F} {\partial Y_{n+1}^{i-1}}$$
is nonzero, vanishes on $f_1,\ldots,f_n,f$ and its last partial derivative applied to $f_1,\ldots,f_n$ does not vanish. So we may assume $i = 1$. The existence of the definable holomorphic function $G$ then follows from the stability of the class of definable functions under the implicit function theorem. 
\end{proof}

\subsection{Proof of Theorem C} Note first that, since the constant field of a differentially closed field $K$ is exponentially algebraically closed, we have that for any holomorphic function $f$ of one variable, $f \in \mathrm{ecl}^{K}(\mathbb{C})$ if and only if $f$ is constant if and only if $f \in \mathrm{hcl}(\mathbb{C})$. This deals with the case $n = 1$ of the following arguments which are proved by induction.

\begin{Prop}\label{o-minimal direction1}
If holomorphic functions of one variable  $f_1,\ldots,f_{n}$ are $\mathrm{hcl}$-dependent over $\mathbb{C}$ then  they are also  exponentially algebraically dependent over $\mathbb{C}$.
\end{Prop} 

\begin{proof} 
Work by induction and assume that the statement holds for some $n$ and consider holomorphic functions $f_1,\ldots,f_{n+1}$ which are $\mathrm{hcl}$-dependent over $\mathbb{C}$. If already $f_1,\ldots,f_n$ are $\mathrm{hcl}$-dependent over $\mathbb{C}$, then the induction hypothesis implies that $f_1,\ldots,f_{n}$ are exponentially algebraically dependent over $\mathbb{C}$. So we may assume without loss of generality that $f_1,\ldots,f_n$ are both $\mathrm{hcl}$-independent over $\mathbb{C}$ and $\mathrm{ecl}$-independent over $\mathbb{C}$. 

Hence, Proposition \ref{prop-from relation to functions} applies and after replacing $U$ by a subdomain, we can write 
\begin{equation}\label{equality-proof-equivalence} f_{n+1} = G \circ (f_1,\ldots,f_n) 
\end{equation}
where $G$ is a definable holomorphic function defined on some $W \subset \mathbb{C}^n$ containing $(f_1,\ldots,f_n)(U)$. Fix $k_0$ a countable field of definition of $G$. By Lemma \ref{lemma-generic-points}, we can find $a \in U$ such that the values $b_1 = f_1(a), ... , b_n = f_n(a)$ are $\mathrm{hcl}$-independent. Hence Fact \ref{fact-Wilkie-complete} (b) applies at the point $(b_1,\ldots,b_n)$ and we can find $G = G_1,\ldots,G_m$ satisfying the conclusion of Fact \ref{fact-Wilkie-complete} (b). Up to restricting the domain $U$ again, the compositions
$$g_i = G_i \circ (f_1,\ldots, f_n) \text{ and } h_i = e^{g_i}$$
are well-defined holomorphic functions on $U$ for $i =1,\ldots,m$. Set 
$$ k = \mathbb{C}(f_1,\ldots,f_n,e^{f_1},\ldots,e^{f_n})^{alg} \text{ and } L = k(g_1,\ldots,g_m, h_1,\ldots,h_m)^{alg}.$$

\begin{claim} 
$\delta(L/k) = 0.$
\end{claim}

\begin{proof}[Proof of the claim]

By Lemma \ref{lemma-exponentialalgebra}, we need to show that $\mathrm{Der}_{\BE}(L/k)$ is the trivial vector space so take $D \in \mathrm{Der}_{\BE}(L/k)$.  Since $G_1,\ldots,G_m$ are implicitly defined from $\mathbb{C}[z,w,e^z,e^w]$ at the point $b$, it follows that we can find polynomials $$P_1,\ldots,P_m \in k[W_1,\ldots,W_m, E_1,\ldots,E_m]$$ such that $P_s(g_1,\ldots,g_m,h_1,\ldots,h_m) = 0$. Viewing the $P \in k[W_1,\ldots, W_m, E_1,\ldots, E_m]$ as an exponential polynomial given by:
$$ P \mapsto \overline{P} =  P(W_1,\ldots,W_n,e^{W_1},\ldots, e^{W_n})  $$
the partial derivative along the $W_i$ can be read as 
$$ \frac{\partial \overline{P}}{\partial W_i} = \overline{\frac{\partial P} {\partial W_i} + E_i \cdot \frac {\partial P} {\partial E_i}}.$$ 
Setting $\partial_i = \frac{\partial} {\partial W_i} + E_i \cdot \frac {\partial} {\partial E_i} \in \mathrm{Der}(k[W_1,\ldots,W_m,E_1,\ldots,E_m]/k)$ the condition on the nonvanishing of the Jacobian in Fact  \ref{fact-Wilkie-complete} (b) can be written as
\begin{equation}\label{equation-nonvanishingoftheJacobian}
 \mathrm{det}\Big( \partial_i(P_s) (g_1,\ldots,g_m,h_1,\ldots,h_m) \Big) \neq 0 \in L.
 \end{equation}
Now since $D$ is a $\BE$-derivation, we have that $D(h_i) = h_i D(g_i)$ so that differentiating $P_s(g_1,\ldots,g_m,h_1,\ldots,h_m) = 0$ we get 
\begin{eqnarray*}
0  & = &  \sum_{i = 1}^m \frac{\partial P_s}{\partial W_i}(g,h) D(g_i) + \sum_{i = 1}^m \frac{\partial P_s}{\partial E_i}(g,h) h_i D(g_i) \\
& = & \sum_{i = 1}^m \Big( \frac{\partial P_s}{\partial W_i}(g,h) + \frac{\partial P_s}{\partial E_i}(g,h) h_i \Big) D(g_i) \\
& = &  \sum_{i = 1}^m \partial_i(P_s)(g,h) D(g_i).
\end{eqnarray*}
Since this is true for any $s = 1,\ldots, m$, the vector $D(g_1),\ldots, D(g_m)$ lies in the kernel of the matrix $(\partial_i(P_s)(g,h))_{i,s}$ which is trivial by (\ref{equation-nonvanishingoftheJacobian}). It follows that $D(g_1) = \cdots = D(g_m) $ so that $D(h_1) = \ldots = D(h_m) = 0$ and $D = 0$ as required.
\end{proof}
To conclude the proof, note that since $f_1,\ldots,f_n$ are exponentially algebraically independent over $C$, $k$ is a self-sufficient subfield of $K$. Since Equation (\ref{equality-proof-equivalence}) implies that $f_{n+1} = g_1 \in L$ and we have shown that $\delta(L/k) = 0$, we conclude that $f_{n+1} \in \mathrm{ecl}^{K}(\C, f_1,\ldots,f_n)$. This means that $f_1,\ldots, f_{n+1}$ are exponentially algebraically dependent over $\mathbb{C}$.
\end{proof} 

\begin{Prop} \label{o-minimal direction2}
Assume that $f_1,\ldots,f_n$ are exponentially algebraically dependent over $\mathbb{C}$ then they are also $\mathrm{hcl}$-dependent over $\mathbb{C}$.
\end{Prop} 

\begin{proof}
Assume that $f_1,\ldots,f_{n+1}$ are exponentially algebraically dependent over $\mathbb{C}$, say $$f_{n+1} \in \mathrm{ecl}^K(\mathbb{C},f_1,\ldots,f_{n}).$$
Working by induction, we may assume that $f_1,\ldots,f_{n}$ are exponentially algebraically independent over $\mathbb{C}$.  Hence, $k = \mathbb{C}(f_1,\ldots,f_{n})^{alg}$ is self-sufficient in $K$. Since $f_{n+1} \in \mathrm{ecl}^K(C,f_1,\ldots,f_{n})$, there exists a finite-dimensional $L/k$ with $\delta(L/k) = 0$. Lemma \ref{lemma-exponentialalgebra} applies and we have that 
\begin{equation} \label{equation}
L \cdot \Theta_{\BE}(L/k)  = \Omega^1(L/k).
\end{equation}
Denote by $r = \mathrm{td}(L/k)$ and consider $(g_i,h_i) \in \Gamma(L)$ generating $\Gamma(L)/\Gamma(k)$ for $i = 1,\ldots, r$. By the Seidenberg embedding theorem, we may assume that $g_i,h_i$ are holomorphic functions on a subdomain $V$ of $U$. Consider the analytic curve 
$$ \gamma: z \mapsto \Big( f_1(z),\ldots,f_{n}(z), g_1(z),\ldots,g_r(z),h_1(z),\ldots, h_r(z)\Big) \in \mathbb{C}^n \times TG_m(\mathbb{C})^r $$
and denote by $Z$ its Zariski-closure. Denote the two projections by 
$$ pr_1 : Z \rightarrow \mathbb{C}^n \text{ and } pr_2 : Z \rightarrow TG_m^r (\mathbb{C}). $$ By construction, $Z$ projects dominantly on the first factor $\mathbb{C}^n$ and the analytic curve $\gamma$ lies on a leaf $\mathcal L_0$ of the foliation $\mathcal F := \mathrm{pr_2}^\ast \mathcal E$
obtained by pullback of the exponential foliation on $TG_m(\mathbb{C})^r$ by the second projection\footnote{In other words, $\mathcal F$ is defined by the vanishing of the one-forms $\theta_i = dg_i/g_i - dh_i$ for $i = 1,\ldots, r$ obtained by pullbacks of the one-forms $\omega_i = dy_i/y_i - dx_i$ on $TG_m^r(\mathbb{C})$.}. We have the following:

\begin{itemize} 
\item [(1)] On a Zariski-open subset of $Z$, the leaves of $\mathcal F$ are locally definable in $\mathbb{R}_{\mathrm{RE}}$. Indeed, the leaves of $\mathcal E$ are locally definable in the o-minimal structure $\mathbb{R}_{\mathrm{RE}}$ and  on a nonempty Zariski-open subset of $Z$, the leaves of $\mathcal F$ are the connected components of the analytic sets of the form $pr_2^{-1}(\mathcal L)$ where $\mathcal L$ is a leaf of $\mathcal E$. 

\item[(2)] The equality (\ref{equation}) says that the foliation $\mathcal F$ is generically transverse to the projection $pr_1$, that is, for $a \in Z$ outside of a proper Zariski-closed subset of $Z$, if $\mathcal L$ is the leaf of $\mathcal F$ through $a$ then 
$$ pr_1 : \mathcal L_a \rightarrow \mathbb{C}^n $$
is locally invertible at $a$. It follows that there is a holomorphic function defined in a neighborhood of $pr_1(a)$ 
$$ F: U_a \rightarrow TG_m^r(\mathbb{C}) $$ and a neighboorhood $W_a$ of $a$ such that  
$$\mathcal L_a \cap W_a := \lbrace (z_1,\ldots,z_n, F(z_1,\ldots,z_n)) \mid (z_1,\ldots,z_n) \in U_a \rbrace.$$
\end{itemize}

Now since by construction, the analytic curve $\gamma$ is Zariski-dense in $X$, we can find $b \in U$ such that both (1) and (2) apply to $a = \gamma(b)$. It follows that the function $F$ obtained in (2) is definable in $\mathbb{R}_{\mathrm{RE}}$ in a neighborhood of $a$. Since $\gamma$ lies on the leaf through $\mathcal L_a$, we conclude that on a neighborhood of $a$
$$ F(f_1(z),\ldots,f_n(z)) = \Big( g_1(z),\ldots, g_r(z),h_1(z),\ldots, h_r(z) \Big) $$
As $f \in \mathbb{C}(g_1,\ldots,g_r,h_1,\ldots,h_r)^{alg}$, using postcomposition by an algebraic function, it follows that $$f_{n+1}(z) = \tilde{F}(f_1(z),\ldots, f_n(z))$$
in a neighborhood of $b$. Hence $f_1(z),\ldots, f_{n+1}(z)$ are not hcl-independent over $\mathbb{C}$ as required.
\end{proof}

Putting together Proposition \ref{o-minimal direction1},  Proposition \ref{o-minimal direction2} and Proposition \ref{prop-from relation to functions}, we obtain the statement of Theorem \ref{theorem-connection-o-minimality}.

\section{The theorems about integration}
In this section, we prove Theorem \ref{intro-thmA} and Theorem \ref{intro-thmB} about the integration of equations which are internal to the constants by exponentially algebraic functions. The preliminary subsections 4.1 and 4.2 are refinements of the results of Rosenlicht \cite{Rosenlicht} on invariant one-forms using the Maurer-Cartan equation. We refer the reader to \cite{BCFN} for the use of Maurer-Cartan forms in the differential-algebraic proof of the Ax-Schanuel Theorem for the $j$-function.

\subsection{The Maurer-Cartan equation} Let $k$ be an algebraically closed field of characteristic zero and let $G$ be a connected algebraic group over $k$ of dimension $n$ and let $\mathfrak{g}$ be its Lie algebra.

\begin{defn} 
An {\em invariant one-form} on $G$ is a rational one-form $\omega \in \Omega^1(k(G)/k)$ which is right-invariant in the sense that $R_g^\ast \omega = \omega$ for all $g \in G$ where $R_g : x \mapsto x \cdot g$ denotes the right-translation by $g$ on $G$.
\end{defn} 
It is well-known that every invariant one-form on $G$ is regular at every point $g$ of $G$ and that the space of invariant one-forms on $G$ is a $k$-vector space of dimension $n$, see \cite{Kolchin-abelian} for example. When $G$ is commutative, a fundamental property is that every invariant one-form is {\em closed}. In general, this condition is replaced by the {\em Maurer-Cartan equation} (\ref{equation-MaurerCartan}) described below.

\begin{defn}
Fix $M$ a smooth algebraic variety over $k$. A {\em rational one-form on $M$ with values in $\mathfrak{g}$} is an element $\theta \in \Omega^1(k(M)/k) \otimes_k \mathfrak{g}$. 
\end{defn} 

Choosing a basis $e_1,\ldots, e_n$ of the Lie-algebra $\mathfrak{g}$ and writing $\mathfrak{g} = ke_1 \oplus \cdots \oplus ke_n$ then we can decompose any rational one-form on $M$ with values in $\mathfrak{g}$ as:
$$\theta = \sum_{i = 1}^n \theta^i \otimes e_i \text{ where } \theta^1,\ldots,\theta^n \in \Omega^1(k(M)/k).$$
We say that $\theta$ is regular at some point $p \in M$ if the one-forms $\theta^1,\ldots \theta^n$ are all regular at $p$. For all the points $p$ in the Zariski-open subset of regular points of $\theta$, the induced map
$$\theta_p : T_p M \rightarrow \mathfrak{g}$$ 
given by $v \mapsto \sum_{i = 1}^n \theta^i(v) \cdot e_i$ is $k$-linear. 

{\defn Assume that $\mathrm{dim}(M) = n$ and let $\theta$ be a rational one-form on $M$ with values in $\mathfrak{g}$. We say that $\theta$ is a {\em rational coparallelism} if for all points $p$ in a Zariski-dense open subset of $M$, $\theta_p: T_p M \rightarrow \mathfrak{g}$ is a $k$-linear isomorphism.}

\begin{exam} \label{example-coparallelism}
Assume that we are given a rational parallelism of $M$, that is, $n$ derivations $D_1,\ldots, D_n$ which form a $k(M)$-basis of $\mathrm{Der}(k(M)/k)$ and denote by $\theta^1,\ldots, \theta^n$ the dual basis  of $\Omega^1(k(M)/k)$. Then 
$\theta = \sum_{i = 1}^n \theta^i \otimes e_i$ 
is a coparallelism on $M$.

Conversely, if a $\mathfrak{g}$-valued one-form $\theta$ is a coparallelism then the one-forms $\theta^1,\ldots, \theta^n$ form a $k(M)$-basis of $\Omega^1(k(M)/k)$. Hence, every rational coparallelism is associated with the rational parallelism of $M$ given by the dual basis $D_1,\ldots, D_n$ of $\mathrm{Der}(k(M)/k)$. 
\end{exam}

\begin{lem} \label{lemma-MaurerCartan-prem}
Denote by $c_{i,j}^l \in k$ the structure constants of the Lie-algebra $\mathfrak{g}$ defined by
$$ [e_i,e_j]_\mathfrak{g} = \sum_{l = 1 }^n c_{i,j}^l \cdot e_l.$$
With the notation of Example \ref{example-coparallelism}, the following are equivalent:
\begin{itemize} 
\item[(i)] for all $1 \leq i,j \leq n$, we have $ [D_i,D_j] = \sum_{l = 1 }^n c_{i,j}^l \cdot D_l$ where $[-,-]$ denotes the Lie-bracket of derivations on $\mathrm{Der}(k(M)/k)$,
\item[(ii)] the coparallelism $\theta$ satisfies the Maurer-Cartan equation
\begin{equation}\label{equation-MaurerCartan} 
d\theta^i = - \frac 1 2 \sum_{j,l = 1}^n c_{j,l}^i \cdot \theta^j \wedge \theta^l \in \Omega^2(k(M)/k).
\end{equation}
\end{itemize}
\end{lem} 

Here, $d: \Omega^1(k(M)/k) \rightarrow \Omega^2(k(M)/k)$ denotes the exterior derivative of one-forms and $\wedge: \Omega^1(k(M)/k) \times \Omega^1(k(M)/k) \rightarrow \Omega^2(k(M)/k)$ the wedge-product of one-forms. See \cite[Section 1]{Blazquez-Casale} and  also \cite[Chapter 3, section 3]{Sharpe} for a proof of this lemma as well as a coordinate-free expression of the equation (\ref{equation-MaurerCartan}).

{\defn Assume that $\mathrm{dim}(M) = n$ and let $\theta$ be a rational one-form on $M$ with values in $\mathfrak{g}$. We say that $\theta$ is a {\em coparallelism of type} $\mathfrak{g}$ if it is a rational coparallelism which satisfies the Maurer-Cartan equation} of Lemma \ref{lemma-MaurerCartan-prem}.

\begin{exam} 
The coparallelism $\theta_G \in \Omega^1(k(G)/k)) \otimes \mathfrak{g}$ associated with the parallelism $D_1,\ldots,D_n$ of $G$ by right-invariant vector fields is a parallelism of type $\mathfrak{g}$. It is the called the (right) {\em Maurer-Cartan form} of $G$.

For $G = \mathrm{GL}_n(k)$ with Lie-algebra $\mathfrak{g} = \mathrm{Mat}_n(k)$. In the coordinates $X_{i,j}$ such that $k[G] = k[X_{i,j}, 1/\mathrm{det}]$, the Maurer-Cartan form of $G$ is given by $$\theta_G = (dX_{i,j})_{i,j} \cdot (X_{i,j})_{i,j}^{-1} \in \Omega^1(G/k) \otimes \mathrm{Mat}_n(k).$$
If $H$ is a connected closed subgroup of $G$ then the pullback $\phi^\ast \theta_G$ of Maurer-Cartan form of $G$ is a $\mathfrak{g}$-valued one-form on $H$ which takes values in the Lie subalgebra $\mathfrak{h} \subset \mathfrak{g}$. The induced $\mathfrak{h}$-valued one-form on $H$ is the Maurer-Cartan form of $H$, see \cite[Chapter 3, Corollary 1.9]{Sharpe}.

As a nonlinear example, if $E$ is the elliptic curve defined by 
$$y^2 = x^3 + ax + b $$
then $\theta = dx/y$ is the Maurer-Cartan form of $E$, see \cite[Chapter 3, Section 5]{Silverman}.
\end{exam}

\begin{lem} \label{lemma-commutativity}
Let $M$ be a smooth algebraic variety over $k$ of dimension $n$ and let $\theta$ be a coparallelism of type $\mathfrak{g}$. Then $\mathfrak{g}$ is abelian if and only if
$$ d\theta = \sum_{i = 1}^n d \theta^i \otimes e_i = 0 \in \Omega^2(k(M)/k) \otimes \mathfrak{g}.$$
\end{lem} 

\begin{proof} 
Denote by $D_1,\ldots, D_n$ the parallelism associated $\theta$. By definition of the wedge-product of one-forms if $s \neq t$ then 
$$ (\theta^j \wedge \theta ^l) (D_s, D_t) = \theta^j(D_s) \theta^l (D_t) - \theta^j(D_t) \theta^l (D_s) = \begin{cases} 1 \text{ if } j = s \text{ and } l = t, \\ -1 \text{ if } j = t \text{ and } l = s, \\ 0 \text{ otherwise.}
\end{cases}$$
It follows from the Maurer-Cartan equation above that for $s \neq t$
$$ d\theta^i(D_s,D_t) = -\frac 1 2 \sum_{j,l} c^i_{j,l} (\theta^j \wedge \theta^l) (D_s,D_t) = -\frac 1 2 (c^i_{s,t} - c^i_{t,s}) = - c^i_{s,t}  $$
Hence $\mathfrak{g}$ is abelian if and only if all the structure constants of $\mathfrak{g}$ are equal to zero if and only if $d\theta^1 = \ldots = d\theta^n = 0 \in \Omega^2(k(M)/k)$, as required.
\end{proof}

\subsection{Abelian invariant differentials} Let $k$ be an algebraically closed field of characteristic zero, let $G$ be a connected algebraic group over $k$ and let $\omega$ be an invariant one-form on $G$. 

\begin{defn} 
Let $L/k$ be a field extension and let $x \in G(L)$. The {\em differential induced by $\omega$ by $x$ on $L/k$} is the one-form denoted $\omega(x) \in \Omega^1(L/k)$ (or $\omega_{L/k}(x) \in \Omega^1(L/k)$ when greater detail is desirable) obtained by pullback of $\omega$ under the (scheme-theoretic) morphism $x : \mathrm{Spec}(L) \rightarrow G.$ 
\end{defn}

We refer to \cite[Chapter II, Section 8]{Hartshorne} for the definition of the scheme-theoretic pullback of one-forms. In the case where $L/k$ is finitely generated, this can be subsume as follows: we can write $L = k(M)$ where $M$ is a smooth algebraic variety over $k$ and identify the point $x \in G(L)$ with a rational morphism 
$$ x : M \dashrightarrow G .$$
The induced one-form $\omega(x) \in \Omega^1(k(M)/k)$ is the pullback of the regular one-form $\omega$ along the rational morphism $x$ of algebraic varieties over $k$. Since any point $x \in G(L)$ is defined over a finitely generated subfield $L_0$ over $k$, the general case can also be defined from the previous construction using the injection of $L$-vector spaces
$$  \Omega^1(L_0/k) \otimes_{L_0} L \rightarrow \Omega^1(L/k).$$
This notion coincides with the notion used by Kolchin in \cite{Kolchin-abelian} and by the second author in \cite{Kirby-semiab}. In particular, the following proposition is extracted from the proof of the Ax-Schanuel theorem for the exponential of semi-abelian varieties in \cite{Kirby-semiab}.

\begin{Prop}\label{proposition-linear}
Let $L/k$ be a field extension, let $A$ be a simple abelian variety and let $G$ be a connected commutative linear algebraic group both defined over $k$. Consider $\omega$ an invariant one-form on $A$ and $\eta$ an invariant one-form on $G$.

Assume that for some $(x,y) \in A(L) \times G(L)$, we have an equality of induced differentials of the form:
$$ \omega(x) + \eta(y) = 0 \in \Omega^1(L/k).$$
Then $\omega(x) = \eta(y)= 0 \in \Omega^1(L/k).$ 
\end{Prop}

\begin{proof} 
An easy reduction shows that it is enough to prove the statement assuming that $k$ has finite transcendence degree and $L/k$ is finitely generated. Embedding $k$ into $\mathbb{C}$, we reduce  to prove the statement for $k = \mathbb{C}$ and $L = k(M)/k$. We may assume that $\omega,\eta \neq 0$ as otherwise the statement is trivial. Note that both forms are closed since the groups are commutative. 

Set $S = A \times G$ and consider the foliation $\mathcal H$ of codimension one on $S$ defined by the vanishing of the closed one-form
$$ \theta = \omega \boxplus \eta := \pi^\ast_A \omega  \oplus \pi^\ast_G \eta \in \Omega^1(S/k).$$
We  can view $(x,y) \in S(L)$ as a rational morphism $\phi = (x,y): M \dashrightarrow S$ which we assume to be regular up to restricting $M$. Since $\omega$ and $\eta$ are invariant forms, we may replace $x$ and $y$ by translation and assume that $\phi(M)$  an irreducible subvariety of $S$ containing the identity element of $S$. 

The equality in the assumption of the lemma states that $$\theta(x,y) = \phi^\ast \theta = 0$$ which means that $\phi(M)$ is contained in a leaf of the foliation $\mathcal H$. Since $e \in \phi(M)$, we conclude that $\phi(M) \subset \mathcal L_e$ where $\mathcal L_e$ is the leaf through the identity element of $G$. It follows from the fact that $\eta$ is an invariant form that $\mathcal L_e$ is an (in general not closed) subgroup of $G$. Zilber's indecomposability Theorem ensures that the (algebraic) group generated by $\phi(M)$
$$ \phi(M) \subset \langle \phi(M) \rangle \subset \mathcal L_e \subsetneq G $$
is a proper {\em connected} algebraic subgroup of $S$. Using that $A$ is abelian and $G$ is linear,  Goursat's lemma implies that  
$$\langle \phi(M) \rangle = B \times H$$ 
where $B$ and $H$ are connected subgroups of $A$ and $G$  respectively. Since the restriction of $\omega$ to $\langle \phi(M) \rangle$ vanishes, we have that $B$ is proper (connected) subgroup of $A$ and, by simplicity of $A$, must be equal to the trivial group.  It follows that $x(M) = \pi_A \circ \phi(M) = \lbrace e_A \rbrace$ and that $\omega(x) = \eta(y) = 0 \in \Omega^1(k(M)/k)$ as stated.
 \end{proof}

\subsection{Relation with differential Galois theory} Assume now that $k$ is an algebraically closed differential field with field of constants $C$ and let $(L,\partial)/k$ be a differential field extension of finite transcendence degree. The Lie derivative: 
$$ \mathcal L_\partial : \Omega^1(L/k) \rightarrow \Omega^1(L/k) $$
defines a differential module structure on $\Omega^1(L/k)$ over $L$.

\begin{lem} 
The differential module $(\Omega^1(L/k),\mathcal L_\partial)$ is the dual of the differential module on $\mathrm{Der}(L/k)$ defined by the Lie bracket with the derivation $\partial$ given by 
$$ D \mapsto  [\partial,D] = \partial \circ D - D \circ \partial.$$
\end{lem}  

\begin{proof} 
First note that since $\partial(k) \subset k$, $D \mapsto  [\partial,D] = \partial \circ D - D \circ \partial$ is well-defined and satisfies the Leibniz rule of differential modules. As defined in \cite{Rosenlicht}, the Lie-derivative is the unique differential module structure on $\Omega^1(L/k)$ satisfying 
$$ \mathcal L_\partial (df) = d (\partial(f)) \text{ for all } f \in L.$$
Therefore it is enough to see that the dual $(\Omega^1(L/k), L_\partial)$ of $(\mathrm{Der}(L/k),[\partial, - ])$ satisfies this condition. Denote by $\langle - , -\rangle: \Omega^1(L/k) \times \mathrm{Der}(L/k) \rightarrow L$ the duality pairing and let $f \in L$. For every derivation $D \in \mathrm{Der}(L/k)$, we have
$$ \partial (\langle df , D \rangle) = \langle L_\partial (df), D  \rangle + \langle df, [\partial, D] \rangle$$
Since $\langle df, D \rangle = D(f)$ for every $D \in \mathrm{Der}(L/k)$, we conclude that 
$$\langle L_\partial(df), D  \rangle = \partial(D(f)) - [\partial, D](f) = D(\partial(f)) = \langle \partial(f), D \rangle.$$
Since this is true for any $D \in \mathrm{Der}(L/k)$, we obtain $L_\partial(df) = d(\partial(f))$ as required. This finishes the proof of the lemma.
\end{proof} 

In the case where $L/k$ is a {\em Picard-Vessiot extension}, we have the following description of the solution sets of these differential modules.

\begin{Prop}\label{proposition-vanderput}
Let $L/k$ be a Picard-Vessiot extension with Galois group $G(C)$. Write $L = k(M)$ for some smooth algebraic variety $M$ over $k$ and denote by $\mathfrak{g}(C)$ the Lie-algebra of $G$ with a $C$-basis $e_1,\ldots, e_n$. Then $M$ admits a coparallelism 
$$ \theta = \sum_{i = 1}^n \theta^i \otimes e_i $$
of type $\mathfrak{g}(k)$ such that $\mathcal L_\partial(\theta^i) = 0$ for $i = 1,\ldots, n$.
\end{Prop}

\begin{proof} 
Note first that standard results of differential Galois theory imply that $\mathrm{td}(L/k) = \mathrm{dim}(G)$ so that $\mathrm{dim}(M) = \mathrm{dim}(\mathfrak{g}(k)) = n$ and that $G$ is connected. 

\begin{fact}[{\cite[Proposition 1.24]{vdP-Singer}}]
Let $L/k$ be a Picard-Vessiot extension. Then 
$$V= \lbrace D \in \mathrm{Der}(L/k) \mid D \circ \partial = \partial \circ D \rbrace$$
is a $C$-Lie-algebra isomorphic to the Lie-algebra $\mathfrak{g}(C)$ of the Galois group $G(C)$ of $L/k$.
\end{fact} 

Denote by $D_1,\ldots, D_n$ the derivations $\mathrm{Der}(k(M)/k)$ corresponding to the chosen basis $e_1,\ldots, e_n$. We first claim that  $D_1,\ldots, D_n$ are {\em $L$-linearly independent}. Indeed, otherwise consider a nontrivial $L$-linear relation with the minimal number of nonzero coefficients written  (up to reordering the $D_i$) as
$$\lambda_1 D_1 + \cdots + \lambda_m D_m = 0 \text{ where } \lambda_i \in L$$
and with $\lambda_1 = 1.$ Applying $[\partial, -]$ and using the Leibniz rule together with $[\partial, D_i] = 0$ leads to 
$$ \partial(\lambda_1) D_1 + \cdots + \partial(\lambda_m) D_m = \partial(\lambda_2) D_2 + \cdots + \partial(\lambda_m) D_m = 0.$$
By minimality of our choice, this must be the trivial $L$-linear relation so that the $\lambda_i$ are constants. This contradicts the $C$-linear independence of $D_1,\ldots,D_n$. 

By comparing the dimension, it follows from the claim that $D_1,\ldots, D_n$ form a $L$-linear basis of $\mathrm{Der}(L/k)$ and therefore a rational parallelism of $M$ of type $\mathfrak{g}(k)$. Hence, they define a coparallelism $\theta$ of type $\mathfrak{g}(k)$ on $M$ by Lemma \ref{lemma-MaurerCartan-prem}. Finally, it follows from the duality that for all $i,j$
$$ 0 = \partial(\theta^i(D_j)) = \mathcal L_\partial(\theta^i)(D_j) + \theta^i([\partial, D_j]) = \mathcal L_\partial(\theta^i)(D_j).$$
Hence, the one-form $\mathcal L_\partial(\theta^i)$ vanishes on the $L$-basis $D_1,\ldots,D_n$ and is therefore identically equal to zero. This finishes the proof of the proposition.
\end{proof}

In the case of an {\em abelian strongly normal extension}, that is, of a strongly normal extension in the sense of Kolchin with Galois group $A(C)$ where $A$ is a (positive-dimensional) abelian variety defined over $C$, we have the following description of the solution set of these differential modules. 

\begin{Prop}\label{proposition-Kolchin} 
Let $L/k$ be an {\em abelian strongly normal extension} with Galois group $A(C)$. There exists an element $\alpha \in A(L)$ and an invariant one-form $\omega$ on $A$ such that the induced one-form satisfies:  
$$ \omega(\alpha) \neq 0 \in \Omega^1(L/k) \text{ and } \mathcal L_\partial(\omega(\alpha)) = 0 \in \Omega^1(L/k).$$ 
\end{Prop}

\begin{proof} 
Denote by $\omega^1,\ldots, \omega^n$ a $C$-basis of invariant one-forms on $A$. Recall the terminology of \cite{Kolchin-abelian}: an $A$-primitive element $\alpha$ is an element of $A(L)$ with the property that:
\begin{equation}\label{equation-Kolchin} \langle \partial, \omega^i_{L/C}(\alpha) \rangle \in k \text{ for } i = 1,\ldots, n
\end{equation}
where $\partial \in \mathrm{Der}(L/C)$ is the derivation on $L/C$, $\langle - , - \rangle : \mathrm{Der}(L/C) \times \Omega^1(L/C)$ is the natural pairing and $\omega^i_{L/C}(\alpha) \in \Omega^1(L/C)$ is the induced one-form at $\alpha$. The following result ensures the existence of nontrivial $A$-primitive elements.

\begin{fact}[{\cite[pp. 782]{Kolchin-abelian}}]
There exists an $A$-primitive element $\alpha$ in $A(L)$ with the property that the extension $L/k \langle \alpha\rangle$ is a finite algebraic extension.
\end{fact}

Denote by $\alpha$ such an $A$-primitive element. Note that the induced one-forms $\omega^i_{L/k}(\alpha) \in \Omega^1(L/k)$  induced by $\alpha$ on $L/k$ are the image of the induced forms $\omega^1_{L/C}(\alpha) \in \Omega^1(L/C)$ on $L/C$ under the projection
$$ \Omega^1(L/C) \rightarrow \Omega^1(L/k)$$ 
Using the expression due to Cartan of the Lie-derivative  on $\Omega^1(L/C)$ based on the exterior derivative as
$$ \mathcal L_\partial = d \circ i_\partial + i_\partial \circ d $$
and the fact that each $\omega^i_{L/C}(\alpha)$ is closed, we conclude that
$$ \mathcal L_\partial(\omega^i_{L/C}(\alpha)) = d \Big( \langle \partial, \omega_{L/C}^i(\alpha)\rangle \Big) $$ 
Since the image of the right-hand side vanishes in $\Omega^1(L/k)$ by Equation (\ref{equation-Kolchin}), it follows from the fact that $\Omega^1(L/C) \rightarrow \Omega^1(L/k)$ is a morphism of differential module for the Lie-derivative that
$$\mathcal L_\partial(\omega^i_{L/k}(\alpha)) = 0 \in \Omega^1(L/k)$$  for $i = 1,\ldots, n$.

It remains to show that for some $i$ and $\omega^i_{L/k}(\alpha) \neq 0 \in \Omega^1(L/k)$. Indeed, assume otherwise. Writing $L = k(M)$ and identifying the point $\alpha \in A(L)$ with a rational morphism
$$ \alpha: M \dashrightarrow A.$$
and denoting by $N$ the Zariski-closure of $M$ under $\alpha$, this implies that the the one-form $\omega^1,\ldots, \omega^n$ vanishes on the tangent space $T_\eta N$ at the generic point $\eta$ of $N$ in $T_\eta A$. Since $\omega^1,\ldots, \omega^n$ are linearly independent at every point of $A$, it follows that $N$ is a point and therefore that $\alpha \in A(k)$. This contradicts that the extension $L/k\langle \alpha \rangle$ is finite and concludes the proof of the proposition.
\end{proof}

\subsection{Binding groups of exponentially algebraic functions} Let $k$ be a self-sufficient algebraically closed  differential field with field of constants $C$ and let $K$ be a differentially closed field extending $k$ without new constants.

\begin{lem} \label{lemma-derivation}
Consider $L/k$ a finite dimensional blurred exponential field extension in $K$ satisfying $\delta(L/k) = 0$. Then  $L$ is a differential subfield of $K$.

\noindent In particular, if $k$ is a differential subfield of $K$ then $\mathrm{ecl}^K(k)$ is also a differential subfield of $K$.
\end{lem}

\begin{proof} 
Since $L/k$ satisfies $\delta(L/k) = 0$, $L$ does not contain any differentially transcendental element over $k$ and is contained in a differential field $M$ of finite transcendence degree over $k$. Consider the exact sequence 
$$ 0 \rightarrow \Omega^1(L/k) \otimes_L M \rightarrow \Omega^1(M/k) \overset{\pi}{\rightarrow} \Omega^1(M/L) \rightarrow 0$$
and the action of the Lie-derivative $\mathcal L_{\partial}: \Omega^1(M /k) \rightarrow \Omega^1(M /k)$. We claim that
$$ \mathcal L_\partial (\Omega^1(L/k) \otimes_L M ) \subset \Omega^1(L/k) \otimes_L M$$
if and only if  $L$ is a differential subfield of $M$.
Indeed, the only thing required to see this is the formula $\mathcal L_\partial (df) = d(\partial(f))$ for $f \in L$. For instance, if $f \in L$ then $df \in \Omega^1(L/k)$ and stability under the Lie-derivative and the previous exact sequence implies that $d(\partial(f)) = 0 \in \Omega^1(M/L)$ which means that $\partial(f) \in L$. The direct implication of the previous equivalence follows and the proof of the converse is similar. Now, since $\delta(L/k) = 0$, Proposition \ref{proposition-predimension} implies that 
\begin{equation}\label{equation-expalg}
L \cdot \Theta_{BE}(L/k) = \Omega^1(L/k).
\end{equation}
Hence, any one-form $\omega \in \Omega^1(L/k) \otimes_L M$ can be written as 
$\omega = \sum_{i = 1}^n \lambda_i \cdot \theta_i$
where $\lambda_i \in M$, $\theta_i \in \Theta_{BE}(L/k)$. This means that each $\theta_i$ is of the form $dy_i/y_i - dx_i$ for some $x_i,y_i \in L$ satisfying $\partial(y_i)/y_i = \partial(x_i)$. Hence
$$\mathcal L_\partial(\omega_i) = \mathcal L_\partial(dy_i/y_i - dx_i) = d (\partial(y_i)/y_i - \partial(x_i)) = 0 $$
Using the Leibniz rule, it follows that $\mathcal L_\partial (\Omega^1(L/k) \otimes_L M ) \subset \Omega^1(L/k) \otimes_L M$ as required. SInce any exponentially algebraic element is contained in an extension of predimension zero, the second clause of the lemma follows.
\end{proof}

\begin{Prop}\label{proposition2-linear}
Consider $q \in S^D(k)$ a $\mathcal C$-internal type which is exponentially algebraic. Then the binding group $\mathrm{Aut}^D(q/\mathcal C)$ of $q$ is isomorphic to a linear algebraic group. 
\end{Prop} 

Here, we say that $q$ is {\em exponentially algebraic type} to mean that $q$ is weakly orthogonal to the constants and that some/any realization of $q$ in a differentially closed field $K$ extending $k$ without new constants belongs to $\mathrm{ecl}^K(k)$. The additional assumption that $q$ is internal to the constants implies that in the saturated model $\mathcal U \models \DCF$ with field of constants $\mathcal C$, the action of $\mathrm{Aut}^D(q/\mathcal C)$ on $q(\mathcal U)$ coincides with the action of a definable permutation group $G(\mathcal U)$ of $q(\mathcal U)$. The binding group theorem \cite[Appendix B]{Hrushovski-Galois}  states that we have a definable isomorphism of permutation groups:
$$ (G(\mathcal U), q(\mathcal U)) \simeq (H(\mathcal C), X(\mathcal C)) $$
where $(H,X)$ is an algebraic permutation group defined over $C$. The content of the proposition is that, under the assumption that $q$ is exponentially algebraic, the algebraic group $H$ is linear.

\begin{proof} For the sake of a contradiction, assume that we can find an exponentially algebraic type whose binding group is not linear in the previous sense. Since $\mathrm{ecl}^K(k)$ is a differential field by Lemma \ref{lemma-derivation}, the main theorem of \cite{Jaoui-Moosa} ensures the existence of an exponentially algebraic type $q \in S(k)$ whose binding group is isomorphic to a simple abelian variety $A$ over the field of constants $C$ of $k$. 

Consider $a \models q$. Since the binding group of $q$ is commutative, $L = k\langle a\rangle/k$ is the strongly normal extension associated to $q$. By Proposition \ref{proposition-Kolchin}, we can find a point $\alpha \in A(L)$ and an invariant differential $\omega$ on $A$ such that the induced differential $\omega(\alpha) \in \Omega^1(L/k)$ satisfies 
$$ \omega(\alpha) \neq 0 \text{ and }  \mathcal L_\partial(\omega(\alpha)) = 0.$$
Now, by exponential algebraicity of $x$ over $k$, we can find a finitely generated $M/k$ satisfying $\delta(M/k) = 0$ and containing $L$. Proposition \ref{proposition-predimension} implies that $M \cdot \Theta_{BE}(M/k) = \Omega^1(M/k)$. Hence $\omega$ can be decomposed as
$$ \omega(\alpha) = \sum_{i = 1}^n m_i \cdot  \theta_i \in \Omega^1(M/k)$$
where $m_i \in M$ and $\theta_i \in \Theta_{\BE}(M/k)$ are $M$-linearly independent. Using that $\mathcal L_\partial(\omega(\alpha)) = 0$, the properties of the Lie-derivative gives: 
$$ 0 = \mathcal L_\partial(\omega(\alpha)) = \sum_{i = 1}^n (\partial(m_i) \cdot \theta_i + m_i \cdot \mathcal L_\partial(\theta_i)) = \sum_{i = 1}^n \partial(m_i) \cdot \theta_i .$$
Since $\theta_1,\ldots, \theta_n$ are $M$-linearly independent, it follows that $m_1,\ldots, m_n \in C$ and hence that 
$$ \omega = \sum_{i = 1}^n \lambda_i \cdot  \theta_i \in \Omega^1(M/k)$$
where the $\lambda_i \in C$ and $\theta_i \in \Theta_{\BE}(M/k)$.
 Expanding each $\theta_i = dy_i/y_i - dx_i$, we can rewrite the previous equality as
$$ \omega(\alpha) = \sum_{i = 1}^n \lambda_i (dy_i/y_i - dx_i) = \sum_{i = 1}^n \lambda_i dy_i/y_i - d \Big(\sum_{i = 1}^n \lambda_i x_i\Big) $$
which is an equality between two induced forms of the form described by Proposition \ref{proposition-linear}. It follows from Proposition \ref{proposition-linear} that $\omega(\alpha) = 0 \in \Omega^1(M/k)$ and therefore in $\Omega^1(L/k)$ too. This is a contradicts our assumption on $\omega(\alpha)$ and finishes the proof. 
\end{proof}

\begin{Prop}\label{proposition-commutative}
Consider $q \in S^D(k)$ a $\mathcal C$-internal type which is exponentially algebraic. Then the binding group $\mathrm{Aut}^D(q/\mathcal C)$ of $q$ is commutative. 
\end{Prop} 

\begin{proof} 
Denote by $L/k$ the strongly normal extension associated to $q$ and fix $\overline{a}$ a fundamental system of realizations of $q$. 
By Proposition \ref{proposition2-linear}, $\mathrm{Aut}^D(L/k)$ is definably isomorphic  to a {\em connected linear} algebraic group $G$ and hence $L/k$ is a Picard-Vessiot extension of $k$. Hence, Proposition \ref{proposition-vanderput} applies and we conclude that $L/k$ admits a coparallelism 
$$ \omega = \sum_{i = 1}^n \omega^i \otimes e_i  $$
of type $\mathfrak{g}(k)$ where $\mathfrak{g}(C)$ is the Lie-algebra of $G(C)$ such that each $\theta^i$ satisfies $\mathcal L_\partial(\omega^i) = 0$.

We claim that $d\omega = 0 \in \Omega^2(L/k) \otimes \mathfrak{g}$, that is, $d\omega^1 = \ldots = d\omega^n = 0$. Indeed, by Lemma \ref{lemma-derivation}, the extension $L/k$ is exponentially algebraic as it is generated by exponentially algebraic elements. Hence, we can find a finitely generated differential field extension $M/k$  without new constants with $\delta(M/k) = 0$ containing $L$. Proposition \ref{proposition-predimension} implies that $M \cdot \Theta_{BE}(M/k) = \Omega^1(M/k)$. Hence each $\omega^i$ can be decomposed as 
$$ \omega^i = \sum_{i = 1}^n m_i \cdot  \theta_i \in \Omega^1(M/k)$$
where $m_i \in M$ and $\theta_i \in \Theta_{\BE}(M/k)$ are $M$-linearly independent. The same argument as in the previous proposition shows that $m_1,\ldots, m_n \in C$. It follows that 
$$d \omega = d\Big(\sum_{i = 1}^n m_i \cdot \omega_i\Big) = \sum_{i = 1}^n m_i \cdot d\omega_i =  0 \in \Omega^2(M/k).$$
Since the pull-back morphism $\Omega^2(L/k) \rightarrow \Omega^2(M/k)$ is injective, it follows that $d\omega^i = 0 \in \Omega^2(L/k)$ for every $i$ as required. Commutativity of $G$ then follows from Lemma \ref{lemma-commutativity}. 
\end{proof}

\subsection{Proofs of Theorem A and Theorem B} Let $k$ be a differential field. Recall that a type $p \in S^D(k)$ is {\em analyzable in the constants} (or $\mathcal C$-analyzable) if it is a definable image of a type $p_n$ obtained from a sequence of definable maps
$$ p_n \rightarrow p_{n-1} \rightarrow \cdots \rightarrow p_1 $$
where at each step the type of the fiber is either algebraic or stationary and internal to the constants. The elimination of imaginaries in the theory $\DCF$ ensures that the this condition is equivalent to the seemingly stronger condition asserting that $p$ itself can be obtained from such a sequence of definable maps. See \cite{Jin} for background on analyzability in the constants.

\begin{theorem}[Theorem A] \label{theoremA}
Let $k$ be a self-sufficient differential field and assume that $p \in S^{D}(k)$ is   a type which is weakly orthogonal to the constants. Then 

\begin{quote}
some/any realization is realized in an elementary differential field extension of $k$ without new constants if and only if $p$ is both analyzable in the constants and exponentially algebraic.
\end{quote}
\end{theorem}
As stated in the introduction, this theorem can be thought as a model-theoretic generalization of a statement of Abel concerning the indefinite integral of algebraic functions: 
\begin{quote} 
if the integral of an algebraic function can be expressed in terms of implicitly or explicitly defined elementary functions then it is elementary.
\end{quote}
Indeed, an equivalent form of Theorem \ref{theoremA} is the following corollary stated in the introduction, from which one recovers Abel's claim using that the equations $y' = a$ are internal to the constants.
\begin{cor} 
Assume that $y(z)$ is an holomorphic solution of an algebraic differential equation
$$F(y,y',\ldots, y^{(n)}) = 0$$
whose coefficients are elementary functions and assume that the equation is {\em internal to the constants}. Then $y(z)$ is exponentially algebraic if and only if it is an elementary function.
\end{cor} 

\begin{proof} 
Since the coefficients $a_0,\ldots,a_n$ of $F$ are elementary functions, $\mathrm{tp}^D(a_0,\ldots,a_n/k)$ is analyzable in the constants. Moreover, by assumption, $\mathrm{tp}^D(y/a_0,\ldots,a_n,k)$ is internal to the constants, it follows by transitivity that $\mathrm{tp}^D(y,a_0,\ldots,a_n/k)$ and hence $\mathrm{tp}^D(y/k)$ are analyzable in the constants. Under this assumption, Theorem \ref{theoremA} states the $y$ is exponentially algebraic if and only if $y$ is contained in an elementary extension of $\mathbb{C}(z)^{alg}$, that is, if and only if $y$ is an elementary function.
\end{proof}

Before proving Theorem \ref{theoremA}, we state the second theorem which is a slightly refined version of Theorem \ref{intro-thmB} from the introduction. Recall that self-sufficient differential fields are algebraically closed by convention.

\begin{theorem}[Theorem B]\label{theoremB}
Let $k$ be a self-sufficient differential field and let $p$ be an exponentially algebraic type which is $\mathcal C$-internal. Then: 
\begin{itemize}
\item[(1)] the binding group $\mathrm{Aut}^D(p/\mathcal C)$ is definably isomorphic to $\Ga^r(C) \times \Gm^s(C)$ for some $r,s \in \mathbb{N}$,
\item[(2)] the type $p$ is interdefinable with a product $$p_1 \otimes \cdots \otimes p_r \otimes  q_1  \otimes \cdots \otimes q_s$$ where the $p_i$ are elementary antiderivatives of elements of $k$ and the $q_j$ are exponentials of elementary antiderivatives,
\item[(3)] in the decomposition given by (2), the realizations of the $p_i$ and of the $q_j$  are respectively of the form 
\begin{equation}\label{equation-theoremB-section4}  
a_i + \sum_{s = 1}^{n_i} \lambda_{i,s} \cdot  \mathrm{ln}(b_{i,s}) \text{ and } e^{a_j} \cdot  \prod_{t = 1}^{m_j} (b_{j,t})^{\lambda_{j,t}}
\end{equation}
where the $\lambda_{i,s}$ and the $\lambda_{j,t}$ are constants and the $a_i,b_{i,s}$ and the $a_j, b_{j,t}$  are elements of $k$.
\end{itemize}
\end{theorem}

Note that for some element $b \in k$, we used the suggestive notation $\mathrm{ln}(b), e^{b}$ and $b^\lambda$ to denote any element of $K\models \DCF$ with the same field of constants as $k$ satisfying respectively
$$ (b,y) \in \Gamma_\BE(K), (y,b) \in \Gamma_\BE(K) \text{ and } y'/y = \lambda \cdot b'/b.$$
Alternatively, an equivalent (differential) form of (\ref{equation-theoremB-section4}) is that for any $x \models p_i$ and any $y \models q_j$, the derivative $x'$ of $x$ and the log-derivative $y'/y$ of $y$ are respectively of the form 
$$ a' + \sum_{ i  = 1}^{n} \lambda_i \cdot \frac {b'_i}{b_{i}}$$
for some $\lambda_1,\ldots, \lambda_n \in C$ and some $a,b_1,\ldots,b_n \in k$. 
When  $p$ is the type of the indefinite integral of an algebraic function $f$ then the binding group of $p$ is a subgroup of the additive group $\Ga$ and  the classical theorem of Liouville \cite{Liouville2} asserts that $\int f dz$ is an elementary function if and only if it is of the special form 
$$ \int f(z) dz = a_0(z) + \sum_{i = 1}^n \lambda_i \cdot \mathrm{ln}(a_i(z))$$
where $\lambda_1,\ldots,\lambda_n$ are constants and $a_1(z),\ldots,a_n(z) \in \mathbb{C}(z,f(z))$. In that case, the characterization obtained from the Theorem \ref{theoremB} is less precise  as (3) only asserts that the $a_i(z)$ live in $\mathbb{C}(z)^{alg}$. Nevertheless, the weaker form given by (3) can not be improved at the level of generality, see \cite{Davenport-Singer}.

We now turn to the proofs of Theorem \ref{theoremA} and Theorem \ref{theoremB} and we first deduce Theorem \ref{theoremA} from Theorem \ref{theoremB}.

\begin{proof}[Proof of Theorem \ref{theoremA}]
The direct implication is immediate: if $p$ is realized in an elementary extension (in the sense of Liouville) of $k$ then $p$ is analyzable in the constants and exponentially algebraic. We prove the converse by induction on the number of steps $l$ of the semiminimal analysis of $p$ which are not algebraic extensions.

If $l = 0$ then the statement is trivial, so assume that the statement is known for all exponentially algebraic types analyzable in the constants in at most $l$ steps for some $l \geq 0$. Consider $q$ an exponentially algebraic type analyzable in $l + 1$ steps and take $a \models q$ in some $K \models \DCF$ extending $k$ without new constants. The subset
$$ S = \lbrace  b \in \mathrm{dcl}^D(k,a) \mid \mathrm{tp}^D(b/k) \text{ is } \mathcal C-\text{internal} \rbrace \subset  \mathrm{dcl}^D(k,a)   $$
is definably closed and not equal to $k = \mathrm{acl}^D(k)$. This follows the existence of a largest internal quotient in the semiminimal analysis of a type in a superstable theory which eliminates imaginaries and the hypothesis that $\mathrm{tp}^D({a}/k)$ is analyzable in the constants. Hence we can pick $b \in S \setminus \mathrm{acl}(k)$ so that $p = \mathrm{tp}^{D}(b/k)$ is $\mathcal C$-internal and not algebraic. By Lemma \ref{lemma-derivation}, every element of $\mathrm{dcl}^D(k,a)$ is exponentially algebraic over $k$ so that $p$ is an exponentially algebraic type as well. Hence, Theorem \ref{theoremB} applies and shows that $p$ is realized in an elementary extension $L$ of $k$. 

Note also that $L$ is a self-sufficient differential field by Example \ref{example-selfsufficient}.  Since, by construction, $\mathrm{tp}^D(a/L)$ is exponentially algebraic and analyzable in at most $l$ steps in the constants, the inductive hypothesis implies that $a$ belongs to some elementary extension of $L$. It follows that $a$ belongs to an elementary extension of $k$. This concludes the proof of the theorem by induction. \end{proof}

\begin{proof}[Proof of Theorem \ref{theoremB}]
(1). Since $k$ is algebraically closed, $\mathrm{Aut}^D(p/\mathcal C)$ is isomorphic to $G(\mathcal C)$ where $G$ is a {\em connected} algebraic group defined over the field of constants $C$ of $k$. Proposition \ref{proposition-commutative} and Proposition \ref{proposition2-linear} further imply that $G$ is a connected commutative linear algebraic group. Hence (1) follows from the structure theory of commutative connected linear algebraic groups \cite{Borel}. 

(2). By (1), the binding group of $p$ is of the form $\Ga^r(C) \times \Gm^s(C)$. Using the Galois correspondence and the vanishing of the cohomology groups $H^1(k,\Gm)$ and $H^1(k,\Ga)$, it follows that $p$ is interdefinable with a type of the form
$p_1 \otimes \cdots \otimes p_r \otimes q_1 \otimes \cdots \otimes q_s$ where the $p_i$ are indefinite integrals of elements in $k$ and the $q_j$ are exponential of such indefinite integrals. To obtain (2), it remains to show that these indefinite integrals belong to elementary extensions of $k$. This is a weak form of (3) so that it is enough to prove (3) to get (2).

(3). Note first that the $p_i$ and the $q_j$ are exponentially algebraic types. This follows from the fact that $p$ is exponentially algebraic and Lemma \ref{lemma-derivation}. We now distinguish between the cases of the $p_i$ and of the $q_j$. 

In the first case, take $t \models p_i$ in some $K \models \DCF$ extending $k$ without new constants. Since $t \in \mathrm{ecl}^{K}(k)$, there exists a finite-dimensional subextension $M/k$ in $K$ such that $\delta(M/k) = 0$. Proposition \ref{proposition-predimension} implies that $M \cdot \Theta_{\BE}(M/k) = \Omega^1(M/k)$. Furthermore, since the derivative $t'$ belongs to $k$, we have that $\mathcal L_\partial(dt) = 0  \in \Omega^1(M/k)$. It follows that we can write 
$$ dt = \lambda_1 \cdot \theta_1 + \ldots + \lambda_n \cdot \theta_n  \in \Omega^1(L/k) $$ 
where $\lambda_1,\ldots, \lambda_n$ belongs to the common field $C$ of constants of $k$ and $K$. Taking such a combination of minimal length, we may assume that the $\lambda_i$ are $\mathbb{Q}$-linearly independent. Writing $\theta_i = dy_i/y_i - dx_i$, \cite[Proposition 4]{Rosenlicht} implies that 
$$d(t - \lambda_1 x_1 + \ldots + \lambda_n x_n) = dy_1/y_1 = dy_n/y_n = 0 \in \Omega^1(L/k).$$
Setting $a = t - \lambda_1 x_1 + \ldots + \lambda_n x_n \in k$ and $b_i = y_i \in k$, we obtain that $$t = a + \sum_{i = 1}^n \lambda_i \cdot x_i = a + \sum_{i = 1}^n \lambda_i \cdot \mathrm{ln}(b_i) $$
since $(b_i,x_i) \in \Gamma_{\BE}(M/k)$ by definition of $\Theta_{\BE}(M/k)$.

The second case can be deduced from the first one: indeed, any realization $t \models q_j$ satisfies $t'/t \in k$. Taking $s \in K$ such that $(t,s) \in \Gamma_{BE}(K)$, we see that $s' \in k$ and it follows from the fact that $tp^{D}(t/k)$ is exponentially algebraic that $tp^{D}(s/k)$ is also exponentially algebraic. Applying the first case to $s$, we obtain that 
$$ t'/t = s' = a + \sum_{i = 1}^n \lambda_i \cdot \frac{b'_i}{b_i}.$$
As noticed after the statement of Theorem \ref{theoremB}, this is equivalent to the integral form given by (\ref{equation-theoremB-section4}). This finishes the proof of the theorem.
\end{proof}

\section{Applications}

\subsection{The problem of integration in finite terms} In the terminology of Bronstein \cite{Bronstein}, the problem of {\em $(A,B)$-integration (in finite terms)} is to determine whether a given expression in the class $A$ has an antiderivative in the class $B$ and to calculate it if it is. We first describe Risch's solution of the classical problem of Liouville where $A=B$ is the class of {\em elementary functions}: there are two obvious number-theoretic obstruction to decidability: 
\begin{itemize}
\item[(1)] we can not decide whether an indefinite integral of the form
 $$\int P(e,\pi) \cdot e^{z^2} dz \text{ with }  P \in \mathbb{Q}[X,Y]$$
is elementary or not since we do not know whether $e$ and $\pi$ are algebraically independent, 

\item[(2)]  indefinite integrals of the form
$$ \int (\mathrm{ln}(e^z) - z) e^{z^2} dz $$
are ambiguous since the answer depends on the chosen branch of the logarithm.
\end{itemize} 
It can be shown that both phenomena provide a source of undecidability for the problem of integration in finite terms (see \cite[Proposition 2.2]{Risch} and the introduction of \cite{Davenport}). To circumvent these problems, we work under the following natural assumptions: 

\begin{Assumption}[Risch] \label{assumption-Risch} \text{ }

\begin{itemize}
\item[(1)] The field of constants $C_0$ is an algebraically closed extension of $\mathbb{Q}$ of finite transcendence degree equipped with an explicit transcendence basis $\mathcal B$. Each constant from $C_0$ will be prescribed by giving its minimal polynomial over $\mathbb{Q}(\mathcal B)$.

\item[(2)] Every elementary function is presented as an element $f$ belonging to an elementary extension $L = C_0(z,\theta_1,\ldots,\theta_n)$ where for each $i \geq 1$ 
\begin{itemize} 
\item either $\theta_i$ is algebraic over $K_{i - 1} = C_0(z,\theta_1,\ldots, \theta_{i-1})$ and its minimal polynomial is explicitly given,
\item or $\theta_i$ is the exponential or the logarithm of an element in $K_{i - 1}$  which does not already have an exponential or a logarithm in $K_{i - 1}^{alg}$.
\end{itemize} 
\end{itemize}
\end{Assumption}
Under these assumptions, one can compute effectively with the elements of $L$: assign distinct symbols to the elements of the field, construct bases for finitely generated subfields, tell when two elements are equal and factor polynomials over it.
Risch's solution of the problem of integration in finite terms asserts that this problem becomes decidable in this setting.

 \begin{fact}[Risch's algorithm]\label{fact-Risch-algo}
Let $L/C_0(z)$ be an elementary differential field extension of $C_0(z)$ satisfying Assumption \ref{assumption-Risch}. Then the problem of integration in finite terms for elements of $L$  is decidable. 
 \end{fact}

We refer to \cite{Singer-Raab} for Risch's proof of this result and to \cite{Bronstein} for another version valid at this level of generality. In the case of algebraic functions, practical implementation of Risch's algorithm are described in Davenport's book \cite{Davenport}. See also  \cite{Pila-Tsimerman} for a description of this algorithm based on the Ax-Schanuel Theorem. These results  provide an important source of (practical) decidability for the problems of exponential algebraicity. For instance, the following corollary states the $(A,B)$-problem of integration where $A$ is the class of elementary functions and $B$ is the class of exponentially algebraic functions is decidable.

 \begin{cor}\label{cor-Rischalgo}
Let $L/C_0(z)$ an elementary differential field extension satisfying Assumption \ref{assumption-Risch}, then 
\begin{quote} 
the problem of integration in finite terms for elements of $L$ by {\em exponentially algebraic functions} is decidable. 
\end{quote}
 \end{cor}
\noindent This is an immediate consequence of Fact \ref{fact-Risch-algo} and Theorem \ref{theoremA}. 

\begin{exam} 
Classical examples of this decision procedure going back to Liouville include the {\em elliptic} (indefinite) {\em integrals} and the {\em Chebychev integrals} given respectively by
$$\int \frac{dz}{\sqrt{z^3 + az + b}} \text{ with } a,b \in C \text{ and } \int z^p (1-z)^q dz \text{ with } p,q \in \mathbb{Q}.$$
They are exponentially algebraic if and only if respectively $4a^3 + 27b^2 = 0$ and at least one of $p,q$ and $p+q$ are integers respectively, see \cite[p.37]{Ritt-book}. Other classical examples from statistics and optics are the {\em error function} and the {\em Fresnel integrals}:
$$ \mathrm{erf}(z) = \frac 2 {\sqrt{\pi}} \int e^{-z^2}dz\text{, } C(z) = \int \mathrm{cos}(z^2)dz\text{, } S(z) = \int \mathrm{sin}(z^2)dz$$
which are all exponentially transcendental.
\end{exam} 

In some cases, we show that decidability of the problem of integration by exponentially algebraic functions remains even after allowing in the integrand other exponentially algebraic functions such as the solutions of Lambert's and Kepler's equations:
$$ w \cdot e^w = z \text{ and } w - \mathrm{sin}(w) = z.$$

\begin{cor} \label{cor-Lambert}
Let $F(w)$ be an elementary function presented as in Risch's algorithm and consider $W(z)$a local inverse of $F(w)$, that is, a solution of the functional equation
$$F(w) = z.$$
The $(A,B)$-integration problem where $B$ is the class of exponentially algebraic functions and $A$ is the set of expressions  of the form $$\int G(z,W(z)) dz$$ 
with $G \in C_0(x,y)^{alg}$ an {\em algebraic function of two variables} is decidable.
\end{cor}

\begin{proof} 
After the change of variable $z = F(w)$ the indefinite integral becomes: 
\begin{equation} \label{eq-intermediate-Lambert}
\int G(F(w),w) \cdot F'(w) dw 
\end{equation}
Following Assumption \ref{assumption-Risch}, we can write $F \in C_0(w, \theta_1,\ldots, \theta_n)$ satisfying the properties (1) and (2) of Assumption  \ref{assumption-Risch}. The element
$$ h(w) = G(F(w),w) \cdot F'(w) \in C_0(w, \theta_1,\ldots, \theta_n)$$ 
belongs to the elementary differential field $L = C_0(w, \theta_1,\ldots, \theta_{n+1})$ where $\theta_{n + 1} = G(F(w),w)$. It follows from Corollary \ref{cor-Rischalgo} that the problem of integration in finite terms for the indefinite integral (\ref{eq-intermediate-Lambert}) by exponentially algebraic functions is decidable. 

It remains to relate exponential algebraicity of (\ref{eq-intermediate-Lambert})  and exponential algebraicity of $ \int G(z,W(z)) dz$ independently of the choice of the solution $W(z)$. This uses the {\em change of variables} formula:
\begin{equation} \label{equation-change of variable} \Big( \int G(z,W(z)) dz \Big) \circ F = \int G (F(w),w) F'(w) dw.
\end{equation} 
We give two different proofs of this result. The first one uses Theorem C and the properties of the o-minimal structures. Restricting if necessary the domain of definitions,  we may assume that
$$ U_w \underset{F}{\overset{W}{\leftrightarrows}} U_z $$
are inverse biholomorphisms where $U_w \subset \mathbb{C}$ and $U_z \subset \mathbb{C}$ are complex disks. On general grounds, the properties of o-minimal structures imply that: 
\begin{itemize}
\item the compositional inverse of an exponentially algebraic biholomorphism is exponentially algebraic, and
\item the class of exponential algebraic functions is stable under composition.
\end{itemize} 
The statement then follows from the formula  (\ref{equation-change of variable}) of change of variables. A  differential algebraic proof accounts of the formula (\ref{equation-change of variable}) as follows: the precomposition $f \mapsto f \circ F$ defines an isomorphism
$$ \mathcal M(U_z)^{alg} \rightarrow \mathcal M(U_w)^{alg} $$
of \textit{blurred exponential fields}. Note that however, this is {\em not a morphism of differential fields}. Setting $i_z = \int G(z,W(z))dz$ and $i_w = \int G(F(w),w) F'(w) dw$, the formula of change of variables implies  that $$\mathrm{tp}^{\BE}(z,i_z/C) = \mathrm{tp}^{\BE}(w,i_w/C)$$
and hence that $\mathrm{etd}(w,i_w/C) = \mathrm{etd}(z,i_z/C)$. Since moreover  $W(z)$ is exponentially algebraic, we also have that $\mathrm{etd}(w,z/C) = 1$. Using the exchange property, we get $$i_w \in \mathrm{ecl}(C,z)\Leftrightarrow i_w \in \mathrm{ecl}(C,w) \Leftrightarrow i_z \in \mathrm{ecl}(C,z).$$ This concludes the proof of the corollary.
\end{proof}

\begin{exam} 
For the classical $W$-function, this algorithm applied to the indefinite $\int \frac{W(z)} z dz$ uses the change of variable $z = we^w$ and computes that
$$ \frac{W(z)} z \cdot dz = \frac w {we^{w}} \cdot e^w(w + 1)  dw  = (w + 1)dw$$ 
admits $w^2/2 + w$ as an indefinite integral and it follows that 
$$ \int \frac{W(z)} z dz = W(z)^2/2 + W(z)$$
which is an expression of the indefinite integral as an exponentially algebraic function. On the other hand, applying the same change of variables to  $\int \frac{W(z)}{z^2} dz$ leads to 
$$ \frac{W(z)} z^2 \cdot dz = \frac w {(w^2e^{2w})} \cdot e^w(w + 1)  dw  = e^{-w} + \frac{e^{-w}} w.$$ 
In that case, the algorithm concludes that the indefinite integral $\int \frac{W(z)}{z^2} dz$ is not exponentially algebraic since Liouville showed that $\int \frac{e^{-w}} w dw$ is not elementary. 
\end{exam}

\subsection{The equation of the pendulum} One of Liouville's motivations to study (pseudo)-elliptic integrals comes from their importance in the resolution of important problems in classical mechanics. In the resolution of the equation of the pendulum:
\begin{equation}\label{equation-pendulum}
\theta'' + \omega^2 \cdot \mathrm{sin}(\theta) = 0
\end{equation}
the elliptic integrals play the role of the inverse trigonometric functions in the resolution of the harmonic oscillator $\theta'' + \omega^2 \cdot \theta = 0$. The solutions of the harmonic oscillator
$$ c_1 \cdot \mathrm{cos}(\omega z) + c_2 \cdot \mathrm{sin}(\omega z) $$
are (obviously) exponentially algebraic.
\begin{cor} \label{cor-pendulum}
Consider the equation (\ref{equation-pendulum}) of the pendulum given by $\theta'' + \omega^2 \cdot \mathrm{sin}(\theta) = 0$ where $\omega$ is a nonzero complex number.  The exponentially algebraic solutions of $(E)$  are precisely the solutions with energy
$$H = \frac 1 2 (\theta')^2 - \omega^2 \cdot \mathrm{cos}(\theta) = \pm \omega^2.$$
\end{cor}

\begin{proof} 
The equation (\ref{equation-pendulum}) and its first integral $H = \frac 1 2 (\theta')^2 - \omega^2 \cdot \mathrm{cos}(\theta) = \pm \omega^2$ are both nonalgebraic. Making the change of variables $u = \mathrm{cos}(\theta/2)$ sends the solutions of $(\ref{equation-pendulum})$ with energy $h_0$ to solutions of the algebraic differential equation: 
\begin{eqnarray} \label{equation-pseudo-elliptic}
 \Big(\frac{du}{dz}\Big)^2 = \omega^2(1-u^2)\Big(u^2 - \frac{\omega^2-h_0}{2\omega^2}\Big).
 \end{eqnarray}
Using the method of separation of variables and solving for $z$ in terms of $u$, we obtain 
$$ z = \int \frac{du}{\omega \sqrt{(1-u^2)(u^2 - \frac{\omega^2-h_0}{2\omega^2})}}$$
which is known as Legendre's elliptic integral of the first kind. Such indefinite integrals were studied by Liouville. He showed (see \cite[p.35]{Ritt-book} and below) that they are elementary if and only if:
\begin{equation}\label{equation-8}
 \frac{\omega^2-h_0}{2\omega^2} \in \lbrace 0,1 \rbrace
 \end{equation} 
which gives $h_0 = \pm \omega^2$. Theorem A states that they are elementary if and only if they are exponentially algebraic and since the class of exponentially algebraic functions is stable under taking local compositional inverse and composition, the corollary follows. 
\end{proof}

The proof presented here is based on Liouville's work and our reformulation of Abel's result on the indefinite integrals of algebraic functions (Theorem A). It was therefore fully accessible at the times of Liouville but we did not find this formulation in Liouville's work. A second proof of this result can be deduced from the results of \cite{Jones-Kirby-Servi}. Finally, a third possibility is a Galois-theoretic proof based on Theorem B presented as Corollary \ref{cor-pendulum-ind} below.  In the Galois-theoretic proof, it is the {\em elliptic equation} (\ref{equation-pseudo-elliptic}) which plays the important role. In that case, the condition (\ref{equation-8}) is interpretated as the set of parameters where the elliptic curve 
$$ v^2 = \omega^2(1-u^2)(u^2 - \frac{\omega^2-h_0}{2\omega^2}) $$
degenerate as a rational curve..

From the first proof, we deduce a general decision procedure for Hamiltonian systems with one degree of freedom.

\begin{cor} \label{cor-Hamiltonian}
Let $V(x)$ be an elementary function described as in Risch's algorithm. Consider the algebraic Hamiltonian system with one degree of freedom
\begin{equation}\label{equation-Hamilton}
 x' = \frac {\partial H} {\partial y}, y' = - \frac {\partial H} {\partial x} \text { where }  H = \frac {y^{2}}{2}  + V(x) 
 \end{equation}
Given $h_0 \in C_0$ in a computable field of constants, one can decide whether all the solutions of the Hamiltonian system with energy $h_0$ are {\em exponentially algebraic} and compute them if they are.
\end{cor} 

\begin{proof} 
The solutions of (\ref{equation-Hamilton}) with energy $h_0$ satisfies the differential equation 
$$ \Big(\frac{dy}{dx}\Big)^2 = h_0 - V(x).$$
Given the potential $V(x)$ as an element $C(x,\theta_1,\ldots,\theta_n)$ satisfying Assumption \ref{assumption-Risch}, one can determine whether $U(x) = \sqrt{h_0 - V(x)}$ lives in $C(x,\theta_1,\ldots,\theta_n)$ If it does, the algorithm applies Risch's decision procedure to the integration of 
$$ \frac{dy} {dx} = \pm U(x)$$
and otherwise applies Risch's decision procedure to $C(x,\theta_1,\ldots,\theta_n,\sqrt{h_0 - V(x)})$ satisfying (1) and (2) in Assumption \ref{assumption-Risch}.
\end{proof}

\subsection{Exponential algebraic independence through differential Galois theory} \label{section-exponential algebraic independence} We now describe applications of Theorem B which treat of the {\em exponential algebraic independence} of several holomorphic functions. 

\begin{cor}\label{cor-elliptic-error-Airy}
The holomorphic functions
$$\mathrm{sn}(z) = \int \frac{dz}{z^3 + az + b},  \mathrm{erf}(z) = \int e^{-z^2} dz \text{ and } \mathrm{Ai}(z) = \frac 1 \pi \int_0^\infty \mathrm{cos}(tz + t^3/3) dz $$
are {\em exponentially algebraically independent} over $\mathbb{C}(z)$.
\end{cor} 

\begin{proof} 
Fix $K$ a differentially closed field containing these three meromorphic functions with field of constants $\mathbb{C}$. Using that $\mathrm{ecl}^K(\mathbb{C},-)$ is a pregeometry, we want to show that 
\begin{eqnarray*}
\mathrm{sn}(z)  \notin \mathrm{ecl}^K(\mathbb{C},z), \mathrm{erf}(z) \notin \mathrm{ecl}^K(\mathbb{C},z, \mathrm{sn}(z)), \mathrm{Ai}(z) \notin  \mathrm{ecl}^K(\mathbb{C},z, \mathrm{sn}(z), \mathrm{erf}(z)).
\end{eqnarray*}
Since the three functions satisfy order two homogeneous linear differential equations, classical computations following Kovacic's algorithm \cite{Kovacic} show that the types $$p = \mathrm{tp}^D(\mathrm{sn}(z)/\mathbb{C}(z)), q = \mathrm{tp}^D(\mathrm{erf}(z)/\mathbb{C}(z)) \text{ and } r = \mathrm{tp}^D(\mathrm{Ai}(z)/\mathbb{C}(z)) $$
are internal to the constants and that their binding groups are respectively the additive group $\Ga(\mathbb{C})$, the group $\mathrm{Aff}_2(\mathbb{C})$ of affine transformations of the affine line and the special linear group $\mathrm{SL}_2(\mathbb{C})$ respectively. 

Two different proofs of the first statement $\mathrm{sn}(z)  \notin \mathrm{ecl}^K(\mathbb{C},z)$ were described after Corollary \ref{cor-pendulum}. We now claim that $\mathrm{erf}(z) \notin \mathrm{ecl}^K(\mathbb{C},z, \mathrm{sn}(z))$. Indeed, set $k = \mathbb{C}(z,\mathrm{sn}(z))^{alg}$ and observe that $k$ is a self-sufficient differential field because $\mathrm{sn}(z)$ is exponentially transcendental, see Example \ref{example-selfsufficient}. Furthermore, Goursat's lemma implies that no nontrivial subgroup of $$\Ga(\mathbb{C}) \times \mathrm{Aff}_2(\mathbb{C})$$
projects surjectively on both factors. This expresses that the $\mathrm{PV}$-extensions generated fundamental solutions of these types are linearly disjoint over $k$ in $K$. It follows that $q$ has a unique extension to $k$, which is its nonforking extension $q \mid k$, and that its binding group is equal to the binding group as $q$ . This means that the binding group of $tp^{D}(\mathrm{erf}(z)/k)$ is $\mathrm{Aff}_2(\mathbb{C})$ and since $k$ is a self-sufficient differential field, Theorem B ensures that $\mathrm{erf}(z) \notin \mathrm{ecl}^K(k) = \mathrm{ecl}^{K}(z,\mathrm{sn}(z))$. 

Finally, we claim that  $\mathrm{Ai}(z) \notin  \mathrm{ecl}^K(\mathbb{C},z, \mathrm{sn}(z), \mathrm{erf}(z))$ follows from a similar proof. Indeed, since $\mathrm{erf}(z)$ is exponential transcendental and its derivative in the exponential of an algebraic function, the differential field $$l = k \langle \mathrm{erf}(z) \rangle^{alg} = \mathbb{C}(z,\mathrm{sn}(z),\mathrm{erf}(z), e^{-z^2})^{alg}$$ is self-sufficient, see Example \ref{example-selfsufficient}. Furthermore, Goursat's lemma implies that no nontrivial subgroup of $$ \mathrm{Aff}_2(\mathbb{C}) \times \mathrm{SL}_2(\mathbb{C}) \text{ nor of } \Ga(\mathbb{C}) \times \mathrm{SL}_2(\mathbb{C}) $$
projects surjectively on both factors. As previously, this means that $r$ admits a unique extension to $l$ and that its binding group is $\mathrm{SL}_2(\mathbb{C})$.  Since $l$ is a self-sufficient differential field and $tp^{D}(\mathrm{Ai}(z)/l)$ as binding group $\mathrm{SL}_2(\mathbb{C})$, Theorem B ensures that $\mathrm{Ai}(z) \notin  \mathrm{ecl}^K(\mathbb{C},z, \mathrm{sn}(z), \mathrm{erf}(z))$ . This finishes the proof. 
\end{proof} 

The second application treats of the solutions of the pendulum.

\begin{cor} \label{cor-pendulum-ind}
Let $h_1,\ldots, h_n \neq \pm \omega^2$ be complex numbers such that the elliptic curves 
$$ (C_i): v^2 = \omega^2(1-u^2)(u^2 - \frac{\omega^2-h_i}{2\omega^2}) $$
are pairwise not isogeneous. Then if $\theta_1,\ldots, \theta_n$ are solutions of the equation of the pendulum (\ref{equation-pendulum}) with respective energies  $h_1,\ldots, h_n$ then $\theta_1,\ldots, \theta_n$ are {\em exponentially algebraically independent} over $\mathbb{C}(z)$
\end{cor} 

\begin{proof} 
Using the same change of variables as in the proof of Corollary \ref{cor-pendulum}, it is enough to prove that $u_1,\ldots,u_n$ given by $u_i = \mathrm{cos}(\theta_i/2)$ are exponentially algebraically independent. The computations following Corollary \ref{cor-pendulum} show that  they are the solutions of elliptic differential equation and the Galois group of $\mathrm{tp}^D(u_i/\mathbb{C}(z))$ is the elliptic curve $(C_i)$ above. 

Using Theorem B, we prove by induction that $k_i = \mathbb{C}(z)(u_1,\ldots,u_i)^{alg}/\mathbb{C}(z)^{alg}$ is a self-sufficient extension of differential fields with exponential transcendence degree $i$. The case $i = 1$ is treated by Corollary \ref{cor-pendulum}. Assume that the result is known for some $i <  n$. Since the $C_i$ are pairwise non isogeneous, Goursat's lemma implies that no nontrivial proper subgroup of
$$ \Big( C_1 \times \cdots \times C_i \Big) \times C_{i + 1} $$ 
projects surjectively on both factors. Hence, $\mathrm{tp}^D(u_{i+1}/k_i)$ has the same binding group as $\mathrm{tp}^D(u_{i+1}/\mathbb{C}(z))$. Using Theorem B, we conclude that $u_{i + 1}$ is exponentially transcendental over $k_i$ so that $k_{i + 1} = k_i(u_i)^{alg}/\mathbb{C}(z)^{alg}$ is a self-sufficient extension of differential fields with exponential transcendence degree $i  + 1$. This concludes the proof of the corollary.
\end{proof}

The third application is an exponential-algebraic version of Kolchin's results \cite{Kolchin} concerning the independence of the Bessel functions
$J_{\alpha_1}(z), \ldots, J_{\alpha_n}(z)$
defined by $$J_\alpha(z) = \sum_{n = 0}^\infty \frac{(-1)^m}{m! \cdot \Gamma(\alpha + m + 1)} (\frac{z}{2})^{2m+ \alpha} $$
and solutions of the linear differential equations \begin{equation}\label{equation-Bessel}  z^2y'' + zy +(z^2 - \alpha^2)y = 0.
\end{equation}
\begin{cor} \label{cor-Bessel}
Let $\alpha_1,\ldots, \alpha_n$ be complex numbers such that 
\begin{center} 
$\alpha_i \notin \mathbb{Z} + 1/2$ and $\alpha_i \pm \alpha_j \notin \mathbb{Z}$ for $i \neq j$. 
\end{center} 
Then the Bessel functions $J_{\alpha_1}(z), \ldots, J_{\alpha_n}(z)$
 and their first derivatives are $2n$ {\em exponentially algebraically independent} holomorphic functions over $\mathbb{C}(z)$. 
\end{cor} 

\begin{proof} 
Work in $K \models \DCF$ a differentially closed with field of constants $\mathbb{C}$ and containing the given Bessel functions. We follow the strategy of Kolchin's proof of their algebraic independence to show that
$$J_{\alpha_1}(z), J_{\alpha_1}'(z)/J_{\alpha_1}(z) \ldots, J_{\alpha_n}(z), J_{\alpha_n}'(z)/J_{\alpha_n}(z)$$ 
are exponentially algebraically independent over $\mathbb{C}(z)$. Define  $k = \mathbb{C}(z)^{alg}$, $u_\alpha = J'_\alpha/J_\alpha$ and $B_\alpha$ such that $(J_\alpha, B_\alpha)$ form a fundamental system of solutions of equation (\ref{equation-Bessel}). The following properties are proved by Kolchin in \cite{Kolchin} using that  $\alpha_i \notin \mathbb{Z} + 1/2$ and that $\alpha_i \pm \alpha_j \notin \mathbb{Z}$ for $i \neq j$:
\begin{itemize}
\item[(i)] for every $\alpha_i$, the Galois group of (\ref{equation-Bessel}) is isomorphic to $\mathrm{SL}_2(\mathbb{C})$,
\item[(ii)] for every $\alpha_i$, $L_i = k(J_{\alpha_i}(z), J'_{\alpha_i}(z), B_{\alpha_i}(z))/k$ is the PV-extension associated to  (\ref{equation-Bessel}) and has transcendence degree three. 
\item[(iii)] for every $\alpha_i$, $M_i = k(u_{\alpha_i}(z))$ is a one-dimensional differential subfield and $\mathrm{Gal}(L_\alpha/M_\alpha)$ is isomorphic the affine group $\mathrm{Aff}_2(\mathbb{C})$,
\item[(iv)] for every $i$, the functions $J_{\alpha_1}(z), J_{\alpha_1}'(z), B_{\alpha_1}(z) \ldots, J_{\alpha_n}(z), J_{\alpha_n}'(z), B_{\alpha_n}(z)$ are algebraically independent over $\mathbb{C}(z)^{alg}$. 
\end{itemize}
We show by induction on $i = 0,\ldots, n$ that 
\begin{quote}
$k_i = k (J_{\alpha_j}(z),J'_{\alpha_j}(z) j = 1,\ldots i)^{alg}$ is a self-sufficient differential field with exponential transcendence degree $2i$. 
\end{quote}  
For $i = 0$, the statement is clear so assume it holds for some $i \geq 0$. Using that  $\mathrm{SL}_2(\mathbb{C})$ is a simple algebraic group, we obtain that
\begin{claim}[Kolchin]
Assume that $L/k$ is a strongly normal extension with Galois group $\mathrm{SL}_2(\mathbb{C})$ and $M/k$ is another one both embedded in a differentially closed field $K$ without new constants. Then either $\mathrm{Gal}(L \cdot M/M) = \mathrm{SL}_2(\mathbb{C})$ or $L \subset M$.
\end{claim}
\begin{proof} 
$\mathrm{Gal}(L \cdot M/M)$ can be identified with a normal subgroup of $\mathrm{Gal}(L/k)$ since $M/k$ is strongly normal. The rest follows from simplicity.
\end{proof} 
Now set $l_i = k_i (B_{\alpha_1}, \ldots, B_{\alpha_i})/k$ which is a PV-extension of $k$. Using (i), (iv) together with the previous claim, we obtain that
$$\mathrm{Gal}(L_{i+1} \cdot l_i/l_i) = \mathrm{SL}_2(\mathbb{C}).$$ 
Now we can apply Theorem \ref{theoremB} to $p = \mathrm{tp}^D(u_{\alpha_{i+1}}/k_i)$ which is internal to the constants with binding group $\mathrm{SL}_2(C)$. It follows that the extension $(M_{i}\cdot k_i)^{alg}/k_i$ is exponentially transcendental and hence that $(M_{i}\cdot k_i)^{alg}$ is a self-sufficient differential field. We can now apply Theorem \ref{theoremB} to $q = \mathrm{tp}^D(J_{\alpha_{i+1}}/(M_{\alpha_{i+1}} \cdot k_i)^{alg})$ with binding group $\mathrm{Aff}_2(C)$ and conclude that $k_{i+1}$ is a self-sufficient differential field with exponential transcendence degree $2(i+1)$. This finishes the proof of the corollary by induction.
\end{proof}

\subsection{Lambert's and Kepler's equations} Finally, we end this article with two algebraic independence statements for nonelementary exponentially algebraic functions. The first one concerns the solutions $w = W(z)$ of Lambert's equation: 
$$ w \cdot e^w = z.$$

\begin{cor} \label{cor-algind-Lambert}
Any distinct solutions $W_1(z),\ldots, W_n(z)$ of Lambert's equation are algebraically independent over $\mathbb{C}(z)$. 
\end{cor}

\begin{proof} 
The functional equation  written as $e^{w} = z/w$ implies that the $W_i(z)$ are solutions of the  $\mathcal L_\Gamma$-formula
\begin{equation}\label{equation-Lambert-Gamma}
\phi(w) := (w,z/w) \in \Gamma 
\end{equation} 
which has coefficients in $\mathbb{C}(z)$. We will show that its solution set is a strongly minimal completely disintegrated set in the theory $\mathrm{DCF}$ which implies the desired algebraic independence statement. 

Indeed, first note the Ax-Schanuel Theorem implies that $\phi(w)$ has no algebraic solutions since $w$ and $z/w$ can not be both constants. Expanding $\Gamma$ explicitly, we obtain the differential equation
\begin{equation}\label{equation-Lambert-differential}
w' = \frac {1}  {(1+1/w)z}
\end{equation}
as an equivalent form of (\ref{equation-Lambert-Gamma}). Since (\ref{equation-Lambert-differential}) has no algebraic solution, this formula isolates a complete type $p \in S^{D}(\mathbb{C}(z)^{alg})$ and Zilber's trichotomy \cite{Hrushovski-Sokolovic} for order one differential equations implies that either $p$ is almost internal to the constants or $p$ is geometrically trivial. 

In the first case, since $p$ is an exponentially algebraic type, it follows from Theorem \ref{intro-thmA} that any realization of $p$ in a field of meromorphic functions is an elementary function. This contradicts Liouville's result that Lambert's functional equation has no elementary solution. We are therefore in the second case and $p$ is geometrically trivial.

It follows that it is enough to prove that any two distinct realizations of $p$ are  algebraically independent over $\mathbb{C}(z)$ to conclude that the solution set of $\phi(w)$ is completely disintegrated. To see this, consider $w_1,w_2 \models p$ which satisfy a nontrivial algebraic relation. Hence,
$$ \mathrm{td}(w_1, z/w_1, w_2, z/w_2 /\mathbb{C}(z)) \leq 1.$$ 
It then follows from the Ax-Schanuel Theorem that $w_1$ and $w_2$ are $\mathbb{Q}$-linearly dependent modulo $\mathbb{C}$ that is
$$q_1 w_1 + q_2 w_2 = c$$ 
for some nonzero $q_i \in \mathbb{Q}$. Taking the derivative and using (\ref{equation-Lambert-differential}), we also obtain 
$$ q_1 \cdot {w_1} \cdot (1 + w_2) + q_2 \cdot {w_2}\cdot (1 + w_1) = 0 .$$
These two equations together lead to 
$$ q_1 w_1 (q_2 + c - q_1 w_1) + q_2\cdot (c - q_1 w_1) (1 + w_1) = 0  $$ 
which must vanishes identically since $w_1$ transcendental over $\mathbb{C}$. It follows that $c = 0$ and $q_1 = -q_2$ so that $w_1 = w_2$ as required. This finishes the proof of the corollary for Lambert's equation.
\end{proof}

An even older example is Kepler's equation (1609) which relates the mean anomaly to the eccentric anomaly in orbital mechanics
$$ w - \mathrm{sin}(w) = z.$$
After writing this equation, Kepler writes:
\begin{quote} 
I am sufficiently satisfied that it cannot be solve a priori, on account of the different nature of the arc and the sine. But if I am mistaken, and any shall point out the way to me, he will be in my eyes the great Apollonius.
\end{quote} 
Kepler's expectation was made precise by Liouville\cite{Liouville} by showing that Kepler's equation admits no elementary solution. 
\begin{cor} \label{cor-Kepler}
Any distinct solutions $W_1(z),\ldots, W_n(z)$ of Kepler's equation are algebraically independent over $\mathbb{C}(z)$. 
\end{cor} 

\begin{proof} 
Following the same line of reasoning, we see that the $W_i(z)$ are solutions of the $\mathcal L_\Gamma$-formula
$$\psi(w) := \exists y, (iw, y) \in \Gamma \wedge w - \frac{1}{2i}(y - 1/y) = z.$$
By quantifier elimination in the theory $\mathrm{DCF}$, this formula is equivalent to a quantifier-free formula in the language $\mathcal L_D$. To compute explicitly a differential equation vanishing on the solutions of Kepler's equation, we take the derivative
$$ w'  - w' \mathrm{cos}(w) = 1$$
and use $\mathrm{cos}^2(w)  + \mathrm{sin}^2(w) = 1$ to obtain that 
\begin{equation}\label{equation-Kepler-circle}
 (1/w' - 1)^2 + (w - z)^2 = 1.
 \end{equation}
After some simplification, we reach 
$$ (w')^2 = \frac{2w'- 1}{(w - z)^2}.$$ 
Since this is an equation of order one and that a nonalgebraic solution is exponentially algebraic,  the generic type $p \in S^D(\mathbb{C}(z)^{alg})$ of this equation is exponentially algebraic. The same reasoning as for Lambert's equation implies that $p$ is a minimal geometrically trivial type.

It is therefore enough to show that any two distinct solutions of Kepler's equation are algebraically independent to finish the proof of the corollary. Consider two algebraically dependent solutions $w_1,w_2$ Kepler's equation. Hence 
$$ \mathrm{td}(iw_1, iw_2, e^{iw_1}, e^{iw_2}/\mathbb{C}(z)) \leq 1 $$
and the Ax-Schanuel Theorem implies that $w_1$ and $w_2$ are $\mathbb{Q}$-linearly dependent modulo $\mathbb{C}$, that is, $q_1 w_1 + q_2 w_2 = c$
for some nonzero $q_1,q_2 \in \mathbb{Q}$. Hence, $q_1 w'_1 + q_2 w_2' = 0$ too. Setting $\eta_j = 1/w'_j$ for $j =1,2$, Equation (\ref{equation-Kepler-circle}) for $j  =1,2$ implies that $(\eta_1, w_1)$ satisfy the two equations:
$$ (\eta_1 - 1)^2 + (w_1 - z)^2 = 1 \text{ and }  \Big(- \frac {q_2} {q_1} \eta_1 - 1\Big)^2 + \Big(\frac{c - q_1w_1}{q_2} - z\Big)^2 = 1 $$ 
Now that these two equations define two quadrics in the plane with coefficients in $\mathbb{C}(z)^{alg}$ given by 
$$ (Y - 1)^2 + (X - z)^2 = 1 \text{ and }  \Big(- \frac {q_2} {q_1} Y - 1\Big)^2 + \Big(\frac{c - q_1 X}{q_2} - z\Big)^2 = 1 $$  
and since $w_1$ is a transcendental function, they must have an infinite intersection. It follows that $q_1 = - q_2$ and $c = 0$ and hence that $w_1 = w_2$ as required.
\end{proof}


\newcommand{\etalchar}[1]{$^{#1}$}

 \end{document}